\documentclass[12pt,a4paper,reqno]{amsart}
\usepackage[utf8]{inputenc}
\usepackage[english]{babel}
\usepackage[OT1]{fontenc}
\usepackage[T1]{fontenc}
\usepackage{amsmath}
\usepackage{amsfonts}
\usepackage{amsthm}
\usepackage{amssymb}
\usepackage{graphicx}
\usepackage{lmodern}
\usepackage{ bbold }
\usepackage{caption}
\usepackage{comment}
\usepackage{leftindex}
\usepackage{hyperref}
\usepackage{titling}
\usepackage{authblk}
\usepackage{mathabx}
\usepackage{color}
\usepackage{esint}
\usepackage{calrsfs}
\usepackage{mathrsfs}
\usepackage{stmaryrd}
 \usepackage{amsaddr}
\usepackage{array}
\usepackage[misc]{ifsym}
 \usepackage[pagewise]{lineno}
\usepackage[toc,page]{appendix}
\usepackage[left=25mm,right=25mm,top=30mm,bottom=20mm]{geometry}

\usepackage{mathtools, stmaryrd}
\usepackage{xparse} \DeclarePairedDelimiterX{\Iintv}[1]{\llbracket}{\rrbracket}{\iintvargs{#1}}
\NewDocumentCommand{\iintvargs}{>{\SplitArgument{1}{,}}m}
{\iintvargsaux#1} %
\NewDocumentCommand{\iintvargsaux}{mm} {#1\mkern1.5mu,\mkern1.5mu#2}

\newcommand{\mres}{\mathbin{\vrule height 1.6ex depth 0pt width
0.13ex\vrule height 0.13ex depth 0pt width 1.3ex}}

\author{Mehdi Badsi$^1$ \and Bastien Grosse$^1$\textsuperscript{\Letter}}
\address{$^1$Nantes Université, Laboratoire de mathématiques Jean Leray, 2 rue de la Houssinière, 44322 Nantes Cedex 3}

\email{Mehdi.Badsi@univ-nantes.fr, bastien.grosse@univ-nantes.fr\textsuperscript{\Letter}}

\newtheorem{theoreme}{Theorem}
\newtheorem{lemme}{Lemma}
\newtheorem{proposition}{Proposition}

\newtheorem{definition}{Definition}

\title{The plasma-wall dynamic: well-posedness}


\numberwithin{equation}{section}
\numberwithin{theoreme}{section}
\numberwithin{lemme}{section}

\numberwithin{proposition}{section}
\numberwithin{algo}{section}
\numberwithin{definition}{section}
\numberwithin{remark}{section}

\DeclareMathOperator{\Div}{div}

\begin{document}
\maketitle

\begin{abstract}
    We prove existence and uniqueness in BV for a Vlasov-Poisson boundary value problem which takes into account a nonlinear boundary condition at the boundary on the electric potential. It models the dynamical interaction between a plasma and an electrically isolated metallic wall. The nonlinear boundary condition is the so-called floating potential equation of plasma physics. 
\end{abstract}

\tableofcontents
\section{Introduction}
Plasmas interacting with material boundaries are ubiquitous in applications, be it in the design of laboratory plasma devices such as Tokamaks, or  in the design of measurement devices such as Langmuir probes \cite{Laframboise,Chen,manfredi_2008}. The understanding of the dynamical interaction between a plasma and a material boundary is therefore fundamental. When a plasma interacts with a metallic boundary, the material acts as a sink of particles \cite{Stangeby_2000}. The material charges itself negatively because of the difference in mobility between the heavy ions and the lighter electrons: electrons are prone to exit the plasma core at a higher rate. As this phenomenon alone would result in a depletion of the plasma, the difference of potential between the material and the plasma core is not arbitrary:  it adjusts itself according to the particle flux that crosses the material boundary. This is  called the floating potential equation (see \cite{Stangeby_2000} section 2.6). When the plasma reaches an equilibrium, the value of the potential at the wall is determined in such a way that  the flux of charges at the material boundary vanishes. This whole mechanism  is called the Debye sheath. Though it is a very active topic of research in plasma physics, the dynamical behavior of the plasma sheath has not been investigated that much in the mathematical community. Some fluids or kinetics models describing the plasma sheath have been studied \cite{GVHKR,Slemrod,SUZUKI2025134743}. They nevertheless do not take into account the floating potential equation which results in an additional nonlinearity in the study of the well-posedness of the problem. The construction of the stationary sheath  solutions \cite{badsi_krm_mybib} along with the study of its temporal stability in a simplified setting \cite{Badsi-25_mybib} were done by Badsi. In this work, we consider the dynamic of the plasma-wall interaction when the potential at the wall satisfies a nonlinear self-consistent equation. We thus model a one dimensional, noncollisional and unmagnetized plasma contained in a chamber of size $L > 0$ made of two species of opposite charges.  Ions  have mass $m{+} > 0$ and charge $q^{+} > 0$ while electrons have mass $m^{-} > 0$ and charge $q^{-} < 0$. Index $+$ is reserved for ions and $-$ is reserved for electrons, and we will use $\pm$ when an expression has to be understood for both indices $+$ and $-$.  Each species of particles is described through its distribution function $f^\pm$ in the phase space $\mathcal{O}:=(0,L)\times\mathbb{R}$. The latter  satisfies the  Vlasov equation: 

\begin{equation} 
\partial_t f^{\pm} + v \partial_x f^{\pm} -\gamma_{\pm} \partial_x \phi \partial_v f^{\pm} = 0 \: \quad \textnormal{ in } \Omega_{T}. \label{m_vlasov+}
\end{equation}

Here, $\Omega_{T} := (0,T) \times \mathcal{O}$ where $T > 0$ is a fixed horizon of time and $\gamma_{\pm} = \frac{q^{\pm}}{m_{\pm}}$ is the charge-to-mass ratio. The self-consistent electric potential $\phi$  obeys the Poisson equation

\begin{equation}
    -\varepsilon_{0} \partial_{xx}\phi = q_+\rho^+ +q_- \rho^-  \textnormal{ in } [0,T)\times (0,L),\label{m_poisson}
\end{equation}

\noindent where $\epsilon_{0}> 0$ is the permittivity,  and

\begin{equation}
    \rho^{\pm}(t,x) = \int_{\mathbb{R}} f^{\pm}(t,x,v) dv \label{m_charge_densities}
\end{equation}

\noindent are the macroscopic densities. To model the interaction between the  plasma   and a metallic wall, we assume that $x= 0$ (the plasma core)  is a source  that injects particles, namely for $t \in (0,T)$

\begin{equation}
    f^{\pm}(t,0,v) = g^{\pm}(t,v) \textnormal{ if } v > 0, \label{m_inject_bc}
\end{equation}

\noindent where $g^{\pm}$  are given nonnegative functions. At $x = L$ (the wall), particles are absorbed, namely for $t \in (0,T)$

 \begin{equation}
    f^{\pm}(t,L,v) = 0 \textnormal{ if } v < 0. \label{m_absorb_bc}
\end{equation} 

As for the electric potential $\phi$, the plasma core at $x =0$ is chosen as the reference of potential and is set to zero, while at the wall the value of the potential adjusts itself according to the  charge current crossing the domain $(0,L)$, that is for $t \in [0,T)$

\begin{equation}
    \phi(t,0) = 0 , \quad \phi(t,L) = \beta + \dfrac{1}{\epsilon_{0}} \int_0^t \int_0^L (q_+j^++q_-j^-)(\tau,x)dx d\tau, \label{m_potential_bc}
\end{equation}

\noindent where $\beta \in \mathbb{R}$ is the initial  value of the potential at the wall, and

\begin{equation}
    j^{\pm}(t,x) = \int_{\mathbb{R}} f^{\pm}(t,x,v) v dv \label{m_current_densities}
\end{equation}

\noindent are the current densities of each species. The system is supplemented with the initial datum

\begin{equation}
    f^{\pm}(0,x,v) = \tilde{f}^{\pm}(x,v), \quad (x,v) \in \mathcal{O} \label{m_initial_datum}
\end{equation}

\noindent where $\tilde{f}^{\pm}$  are given  initial distribution functions.  The set of equations \eqref{m_vlasov+}-\eqref{m_initial_datum} will be referred to as the Vlasov-Poisson-Ampère system (VPA).

 The main novelty is the non linear boundary condition \eqref{m_potential_bc} on the potential which states that the potential at the wall evolves freely in response to the current of incoming particles. Because of the boundary conditions \eqref{m_inject_bc}, \eqref{m_absorb_bc}, the solutions of the Vlasov equations are rarely classical \cite{guo_1994}. They are only of bounded variations. Our main result establish the well-posedness of (VPA) with BV regularity and bounded velocity moments for the distribution functions under assumptions on the data which are reminiscent of the ones made in previous works \cite{Guo,guo_1994,Filbet-Guo}, though we do not assume  compactly supported data in velocity.  The proof proceeds in three main steps. We first establish the BV regularity of solutions to the Vlasov equation supplemented with the boundary conditions \eqref{m_inject_bc} and \eqref{m_absorb_bc}. To this end, we employ a finite-volume upwind scheme and closely follow the approach developed in \cite{aguillon}. Since the velocity variable is unbounded, several modifications are required, including the choice of a suitable CFL condition and the derivation of uniform bounds on the velocity moments. The second step is based on a fixed-point argument for the electric field. We prove the existence and uniqueness of solutions to a standard Vlasov--Poisson system with a prescribed, time-dependent potential at $x=L$ (see Section \ref{sec:Vlasov-Poisson}). Although this part essentially follows the method of \cite{Guo}, we reproduce each estimate as they crucially depend on the wall potential. The third and final step relies on a fixed-point argument for the wall potential $\phi(t,L)$. We first prove the existence and uniqueness of a fixed point, and hence of a solution to the (VPA) system, on a sufficiently short time interval. We then derive a uniform bound on the wall potential, allowing us to extend the local solution to arbitrary times. This extension is made possible by an energy estimate. The key observation is that the floating-potential boundary condition \eqref{m_potential_bc} is precisely the condition that allows the total energy of the system to be controlled in terms of the initial energy and the energy carried by the incoming particles. Finally, we show that the total energy provides an upper bound for any potential difference and thus yields a bound on the wall potential.

The paper is organized as follows. We begin by introducing the notation, recalling the necessary results on functions of bounded variation, and stating our main theorem. In Section \ref{sec:linear-Vlasov}, we derive BV estimates for the linear Vlasov equation in a bounded domain. In Section \ref{sec:Vlasov-Poisson}, we investigate the Vlasov--Poisson system with Dirichlet boundary conditions imposed on the electrostatic potential. In particular, we establish the propagation of both moment and BV estimates for the distribution function. Finally, Section \ref{sec:Vlasov-Poisson-floating} is devoted to the proof of our main result.
\section{Notations, functions of bounded variation and main result}

  \label{sec:BV_space}

We will intensively use the notations for the following sets, with $T\in\mathbb{R}^{+\star}$: $\mathcal{O}=(0,L)\times\mathbb{R}$, $\Omega_T=(0,T)\times\mathcal{O}$ and $Q_T = [0,T]\times[0,L]$. In  the sequel, we will  often rely on a cutoff function $\chi_R$ built as following. First,  we shall select a smooth, non-decreasing  function  $\alpha : \mathbb{R}\mapsto [0,1]$  such that $\alpha(x)=0$ if $x\leq 0$ and $\alpha(x)=1$ if $x\geq 1$. For $R > 0$, we then define the cutoff function $\chi_R$ by 
 $$\chi_R(x) =
\left\lbrace \begin{array}{ccc}
    \alpha(x+R+1)  & \mbox{ if } x \ \leq 0,  \\
     \alpha(-x+R+1 ) & \mbox{ if } x \ \geq 0.
 \end{array}\right.
 $$

 The $d$-dimensional Hausdorff measure is denoted by $\mathcal{H}^d$. For a given positive measure $\mu$ on $\mathbb{R}^d$, the mean of a $\mu$-measurable function $f$ on the	  $\mu$-measurable set $A$ of finite measure  is defined by

 \begin{equation}
 \fint_A f d\mu = \frac{1}{\mu(A)} \int_A f d\mu,
\end{equation}  

\noindent whenever the last integral makes sense. With a slight abuse of notation, if $f:\Omega_T\mapsto \mathbb{R}$ is a real-valued function,  we denote by $f(t)$ its restriction to the hyperplane $\{t\}\times\mathcal{O}$:

$$
\begin{array}{cccc}
f(t):&(x,v)&\mapsto &f(t,x,v) \\
     & \mathcal{O}&\mapsto& \mathbb{R}
\end{array}
$$

 For an open subset $U\subset \mathbb{R}^{n}$, the space of bounded variation functions on $U$ is denoted $BV(U).$ Here is a first definition of this space.

\begin{definition}[\cite{evansgariepy}, section 5.1]
    Let $U$ an open subset of $\mathbb{R}^n$. Let $f\in L^1 (U)$. We say that $f$ is of bounded variations if the total variation $TV(f,U)$ of $f$ in $U$ is finite:

    $$TV(f,U) :=\sup\left\{  \int_{U} f \Div\varphi dx \  | \ \ \varphi\in (C_{0}^{\infty}(U))^n, \|\varphi\|_{L^{\infty}(U)}\leq 1 \right\} <\infty. $$

     The space of functions of bounded variations is noted $BV(U)$ and can be endowed with the norm $\| f \|_{BV(U)} = \|f \|_{L^1(U)} + TV(f,U)$.
     
     The space $BV_{loc}(U)$ contains the functions which belong to $BV(V)$ for every open set $V$ such that $V\subset U$.
\end{definition}

The set of bounded Radon measures on an open set $U\subset \mathbb{R}^n$ is noted $\mathcal{M}(U)$. An  equivalent definition for the space $BV(U)$ is the following:

\begin{definition}[\cite{ambrosio}, pages 117-120]
    Let $U\subset \mathbb{R}^n$ be an open set. A function $f$ lies in  $BV (U)$ if there exist $n$ bounded Radon measures $\mu_{1},\mu_2,...,\mu_n$ such that for all $0\leq i\leq n$ and $\varphi\in C_{0}^{\infty}(U)$, 

    $$
    \int_{U} f(x) \partial_{x_i}\varphi(x) dx = -\int_{U} \varphi(x) d\mu_i (x).
    $$
    
    We note $ \partial_{x_i}f :=\mu_i $, and the distributional gradient is  $Df = (\partial_{x_1} f,...,\partial_{x_n} f)$. Finally,  the total variation of $TV(f,U)$ equals the total variation of the measure $Df$, in other words $|Df|(U) = TV(f,U)$.
    
\end{definition}

An important property of the space $BV_{loc}$ is a compactness property.

\begin{theoreme}[\cite{ambrosio}, theorem 3.23]
    Every family of functions $(f_h)_{h>0}\in BV_{loc}(U)$ such that 

    $$
    \sup_{h>0} \left\{ \int_A f_h dx + |Df_h|(A) \ \ \lvert \ A \subset U, A \ open\right\} <\infty
    $$
    admits a subsequence $(f_{h_{n}})_{n\in\mathbb{N}}$ which converges in $L^1_{loc}(U)$ toward $f\in BV_{loc}(U)$.
    \label{theoreme/compacite}
\end{theoreme}

By definition, BV functions satisfy an integration by parts formula when tested against smooth vector fields  vanishing at the boundary. In fact, this formula can be generalized for $C^1$ vector fields, and even define the trace of a BV function on the boundary.

\begin{theoreme}[\cite{evansgariepy}, section 5.3] \label{divergence_formula}
    Let $U\subset\mathbb{R}^n$ be an open, bounded set with $\partial U$ lipschitz. Let  $\overrightarrow{n}$ be the normal  on $\partial U$.  Then there exists a unique bounded linear operator $T : BV(U) \mapsto L^1(\partial U,d\mathcal{H}^{n-1})$ such that the following equality holds for all $\varphi\in (C^1(U))^n$:

    \begin{equation}
        \int_U f \Div\varphi dx = -\int_U \varphi \cdot d(Df) + \int_{\partial U} (\varphi \cdot \overrightarrow{n}) Tf d\mathcal{H}^{n-1} .    \label{equation/green}
    \end{equation}

    $Tf$ is called the trace of $f$ on $\partial U$. 
\end{theoreme}

The trace $Tf$ admits an expression in terms of local means of $f$ with respect to the Lebesgue measure, as shown by the following result:

\begin{theoreme}[\cite{evansgariepy}, theorem 2 p.181]

If $U$ is open, bounded with Lipschitz boundary, then for almost all $u\in \partial U$, 

$$
 Tf(u) = \lim\limits_{R\to 0^+} \fint_{U\cap B(u,R)} f(y) dy.
$$
\label{theoreme/trace}

\end{theoreme}

Theorems \ref{divergence_formula} and \ref{theoreme/trace} are stated for bounded domains but they also admit an extension to our setting on the unbounded domain $\Omega_T$ and for special test fields. The proof of the following result is the appendix.
\begin{theoreme} \label{theoreme:trace}
Fix $p\in\mathbb{N}$ and consider a nonnegative function $f\in L^\infty(\Omega_T) $ such that $|v|^kf \in BV(\Omega_T)$ for all $k\in\Iintv{0,p}$. Then $T(|v|^k f) = |v|^k Tf\in L^1(\partial\Omega_T)$, with the trace $Tf$ defined as in Theorem \ref{theoreme/trace}.  Moreover, we have that for all $(\phi_1,\phi_2)\in C^\infty(Q_T)$,

\begin{multline}
\int_{\Omega_T} f v^k  \Div(\phi_1,\phi_2,0) dtdxdv = \int_{\partial\Omega_T} Tf v^k(\phi_1,\phi_2,0)\cdot \overrightarrow{n} d\mathcal{H}^2\\-\int_{\Omega_T} v^k(\phi_1,\phi_2,0)\cdot d(Df).
\end{multline}

\label{theoreme/regularite_trace}
\end{theoreme}

The derivative $Du$ of a  function $u\in BV(\Omega_T)$ can be decomposed as the sum of the diffuse part $\tilde{D}u$ and of the jump part $D_{jump} u$. More precisely, one can define a jump set $J_u$, a normal vector field $\overrightarrow{n}_u$ which orients $J_u$, and values $u_+$ and $u_-$ which are limits of means taken on half-balls on both sides of $J_u$. The jump part takes the form 

$$
D_{jump} u = (u_+ - u_-) \overrightarrow{n}_u \otimes \mathcal{H}^2\mres J_u.
$$

The diffuse part $\tilde{D}u$ is the sum of an absolutely continuous part with respect to the Lebesgue measure, and of a Cantor part. We will not need such degree of detail and refer to \cite{evansgariepy,ambrosio} for details.  We say that $u$ has an approximate limit at  $x\in U$,  if there exists a real number  $\tilde{u}(x)$  such that

	$$
	\lim\limits_{r\to 0^+ }\fint_{B(x,r)} |u(y)- \overline{u}(x)  | dy .
	$$ 
    
    If it exists, it is unique. The approximate limit exists almost everywhere in $U$ for functions of bounded variations. The following chain rule for BV functions expresses the derivative of $f\circ u$, when $f$ is Lipschitz continuous 
\begin{theoreme}[\cite{ambrosio}, theorem 3.99] \label{m_chain_rule}
    Let $u\in BV(U)$ and $f:\mathbb{R}\mapsto\mathbb{R}$ be a Lipschitz continuous function satisfying $f(0)=0$ if $|U|=\infty$. Then $v= f\circ u$ belongs to $BV(U)$ and the derivative of $v$ can be decomposed as the sum  $Dv = \tilde{D}v + D_{jump} v$, where

    $$
    \tilde{D}v = f'(\overline{u})\tilde{D}u,
    $$

    $$
    D_{jump} v = \dfrac{f(u_+)-f(u_-)}{u_+-u_-} D_{jump}u.
    $$
    
	The derivative of $f$ has to be understood in the classical sense, and exists almost everywhere for Lipschitz continuous functions.    
    \label{theoreme/chainrule}
\end{theoreme}

Finally, we will have to disintegrate the partial derivative $\partial_v f$ of functions $f$ in $BV(\Omega_T)$. By using the construction and the remark in (\cite{ambrosio}, page 204), we get the following result.

\begin{theoreme}

Let $f\in BV(\Omega_T)$. Then $\partial_v f = dt \otimes (\partial_v f(t)) $ and $\partial_x f = dt \otimes (\partial_x f(t)) $, where $(\partial_v f(t)), (\partial_x f(t))$ are the partial derivative of the restriction 

$$
f(t):(x,v)\mapsto f(t,x,v)
$$

of $f$ to the hyperplane $\{t\}\times\mathcal{O}$. Their total variations are given by $|\partial_x f| = dt \otimes |(\partial_x f(t))| $ and $|\partial_v f| = dt \otimes |(\partial_v f(t))| $, which implies the following equality:

$$
|\partial_x f|(\Omega_T) +|\partial_v f|(\Omega_T)  = \int_0^T TV(f(t),\mathcal{O}) dt.
$$
\label{theoreme/desintegration}

\end{theoreme}

We are now ready to state our main result which is the following.
\begin{theoreme}[Existence and uniqueness of weak solution for (VPA)] \label{MAIN_RESULT}   Assume, for an integer  $p \geq 4$, that 
\begin{itemize}
        \item $|v|^k\tilde{f}^\pm \in  BV(\mathcal{O}), k=0,...p$,
        \item $|v|^k g^\pm\in BV((0,T)\times\mathbb{R}^+), k=0,...p+1$,
        \item $|v|^k \tilde{f}^\pm\in L^\infty(\mathcal{O}),|v|^k g^\pm\in L^\infty((0,T)\times\mathbb{R}^{+\star})$, $k=0,...p$,
        \item $\underset{ t \in (0,T)}{\sup} \underset{ \Delta t > 0}  \sup \dfrac{1}{\Delta t}|\partial_v(|v|^kg^{\pm})|((t,t+\Delta t)\times\mathbb{R}^{+\star})< + \infty,$  $k=0,...,p$,
        \item $\underset{ t \in (0,T)}{\sup} \underset{ \Delta t > 0}  \sup \dfrac{1}{\Delta t}|\partial_t(|v|^kg^{\pm})|((t,t+\Delta t)\times\mathbb{R}^{+\star})< + \infty,$ $k=0,...,p$,
        \item $\underset{ t \in (0,T)} \sup  \int_{\mathbb{R}^{+}} |v|^k g^\pm(t,v)dv <\infty$, $k=0,...,p$.
\end{itemize}

Then, the Vlasov-Poisson-Ampère system  \eqref{m_vlasov+}-\eqref{m_initial_datum}  has a unique weak solution \\ $(f^{+},f^{-}, \phi) \in C^{0}\Big([0,T]; L^{1}(\mathcal{O}) \Big)^2 \times W^{2,\infty}(Q_{T})$ such that for all   $ k \in \Iintv {0,p-1}$, \\ $|v|^{k} f^{\pm} \in BV(\Omega_{T})$ and that for all $k \in \Iintv {0,p}$ and almost all $t$,  $|v|^{k} f^{\pm}(t) \in BV(\mathcal{O})$. 
\end{theoreme}

\section{ The  Vlasov equation with boundary conditions} \label{sec:linear-Vlasov}
\subsection{Weak solutions and main result of the section}

Let $T_r <T_{r+1}$ be two nonnegative numbers such that $[T_r,T_{r+1}]\subset[0,T]$. We set for ease in the reading $\Omega_r = (T_r,T_{r+1})\times \mathcal{O}$ and $Q_r= [T_r,T_{r+1}]\times [0,L]$. In this section, we study the linear Vlasov equation in bounded domain with a given electric field $E\in W^{1,\infty}\big( Q_{r} \big)$. Since the sign of the charge plays no role in the analysis, the problem amounts to study the regularity of the solution $f$ to the  Vlasov equation with Dirichlet  boundary conditions: 
\begin{equation} 
\left\lbrace
\begin{array}{ll}
\partial_t f + v \partial_x f +\gamma E \partial_v f = 0 &  \textnormal{ in } \Omega_{r}, \\ 
f(t,0,v) = g(t,v) &\text{ if } (t,v) \in (T_{r},T_{r+1}) \times \mathbb{R}^{+\star}, \\
f(t,L,v) = 0 &\text{ if } (t,v) \in (T_{r},T_{r+1}) \times \mathbb{R}^{-\star}, \\ 
f(T_r,x,v) = \tilde{f}(x,v) & \text{ if }  (x,v)\in \mathcal{O},  \\
\end{array}\right.
\label{systeme/vlasovlineairebord}
\end{equation}
Here,  $\gamma \neq 0$ is a given physical parameter, $g : (T_{r},T_{r}) \times \mathbb{R}^{+\star} \longrightarrow \mathbb{R}^{+}$ and $\tilde{f} : \mathcal{O} \longrightarrow \mathbb{R}^{+}$ are given non-negative functions. Due  to the boundary conditions, the solutions of the Vlasov problem \eqref{systeme/vlasovlineairebord} are rarely smooth \cite{guo_1994}. Without extra-compatibility conditions on the datum, it is known that some discontinuity at the boundary can occur and be propagated along a characteristic curve of the equation. 
We shall therefore consider weak solutions for the Vlasov equation.  The following theorem which is due to Boyer (\cite{boyertrace,aguillon}) provides existence and uniqueness of the weak solution when data are compactly supported.

\begin{theoreme}[Existence and uniqueness of weak solutions \cite{aguillon}]  \label{m_def_weaksol_vlasov}
Assume that $E\in W^{1,\infty}(Q_r)$,  $\tilde{f}\in L^\infty(\mathcal{O})$, $g\in L^\infty((T_r,T_{r+1})\times\mathbb{R}^{+\star})$, and that $\tilde{f},g$ are compactly supported. There exists a unique function $f\in L^\infty(\Omega_r)\cap C^0([T_r,T_{r+1}], L^1(\mathcal{O}))$, and a function $\overline{T}f\in L^\infty((T_r,T_{r+1})\times \partial \mathcal{O})$  which satisfies 

$$
f(T_r,.,.) = \tilde{f} \mbox{ , } \overline{T}f = g \mbox{ in } (T_r,T_{r+1})\times\mathbb{R}^{+\star} \mbox{ and } \overline{T}f = 0 \mbox{ in } (T_r,T_{r+1})\times\mathbb{R}^{-\star},
$$

\begin{multline} \label{m_weak_form}
\int_{\Omega_r} (\partial_t \varphi + v \partial_x \varphi + \gamma E \partial_v \varphi) f dtdxdv + \int_{T_r}^{T_{r+1}} \int_{\mathbb{R}} v \overline{T}f(t,v) \varphi(t,0,v) dv dt \\- \int_{T_r}^{T_{r+1}} \int_{\mathbb{R}} v \overline{T}f(t,v) \varphi(t,L,v) dv dt  + \int_{\mathcal{O}}\tilde{f}(x,v) \varphi(T_{r},x,v) dvdx \\-   \int_{\mathcal{O}}f(T_{r+1},x,v) \varphi(T_{r+1},x,v) dvdx  = 0,
\end{multline}

for all $\varphi \in W^{1,1}(\overline{\Omega_r})$ with bounded support.

\end{theoreme}

The goal of this section is to prove BV regularity for the weak solution. This will be done using a converging upwind finite volume scheme. First, BV estimates will be established directly on the scheme for compactly supported data. As the scheme converges, we will obtain the estimates on the weak solution. Then, we will extend the result for data with unbounded supports. This is not trivial since the ratio $\frac{h}{\Delta t }$ depends linearly on the size of the supports of the data and thus needs a  suitable control. We also impose higher regularity on the data to control moments in velocity.  With the notations introduced above, we shall establish the following.

\begin{proposition}[Regularity for the  Vlasov equation] Assume, for an integer $p\geq 1$, that

    \begin{itemize}
        \item $|v|^{k} \tilde{f}\in BV(\mathcal{O}) \cap L^\infty(\mathcal{O})$ for $k=0,...p$,
        \item $|v|^{k} g\in BV((0,T)\times\mathbb{R}^{+\star}) \cap L^\infty((0,T)\times\mathbb{R}^{+\star})$ for $k=0,...p$, 
        \item $|v|^{p+1} g\in BV((0,T)\times\mathbb{R}^{+\star})$,
        \item $\underset{ t \in (0,T)} \sup \underset{ \Delta t > 0} \sup \dfrac{1}{\Delta t}|\partial_v(|v|^kg)|((t,t+\Delta t)\times\mathbb{R}^{+\star})<\infty$ for $k=0,...,p$,
        \item $\underset{ t \in (0,T)} \sup \underset{ \Delta t > 0} \sup \dfrac{1}{\Delta t}|\partial_t(|v|^kg)|((t,t+\Delta t)\times\mathbb{R}^{+\star})<\infty$ for $k=0,...,p$,
        \item $\underset{ t \in (0,T)} \sup  \int_{\mathbb{R}^{+}} |v|^k g^\pm(t,v)dv <\infty$ for  $k=0,...,p$.
    \end{itemize}
    
    Then the  Vlasov boundary value problem \eqref{systeme/vlasovlineairebord} admits a unique weak solution $f \in C^0\Big([0,T]; L^{1}(\mathcal{O}) \Big)$ such that $|v|^k f\in BV(\Omega_T)$ for $k \in \Iintv{0,p-1}$, and that for almost all $t$ and all $k \in \Iintv{0,p}$,    $|v|^k f(t)\in BV(\mathcal{O})$. \label{proposition:regularite_vlasov}
\end{proposition}

\subsection{Properties of weak solutions with BV regularity}

The next proposition shows that the continuity equation is true under mild regularity hypothesis. It also establishes the  regularity of $\rho$ and $j$. These results will be helpful when we address the Vlasov-Poisson system as we need estimates on the electric field.

\begin{proposition}[Continuity equation and of regularity of the two first moments]
Let $f$ be a weak solution in the sense of definition \ref{m_def_weaksol_vlasov}.  
    \begin{enumerate}
        \item If $f, |v|f \in L^{1}(\Omega_r)$, then $\rho$ and $j$ satisfy the continuity equation in the space of distribution $\mathcal{D}'(Q_{r})$:
        \begin{align} \label{m_continuity_equation}
          \partial_{t}\rho = - \partial_{x}j . 
        \end{align}        
        \item If moreover $f\in BV(\Omega_r)$, then $\rho\in  BV(Q_{r})$, and the continuity equation \eqref{m_continuity_equation} is true in $\mathcal{M}(\Omega_r)$.
        \item If $ |v|f \in BV(\Omega_r)$, then $j\in  BV(Q_{r})$.
    \end{enumerate}

\label{proposition/regulariterhoj}
\end{proposition}

\begin{proof}

Let $\varphi \in C_{0}^{\infty}(Q_{r})$. We take the test function $\psi(t,x,v) = \chi_{R}(v)\varphi(t,x)$ in the weak formulation \eqref{m_weak_form} to get

$$
\int_{T_r}^{T_{r+1}}\int_{0}^{L}\int_{\mathbb{R}}f (\chi_{R}\partial_t \varphi + v\chi_{R} \partial_x \varphi + \gamma E \varphi  \partial_v \chi_{R}) dvdxdt = 0.
$$

By the dominated convergence theorem, as $\partial_v \chi_{R}\to 0$ when $R\to \infty$, and that $\chi_R f \to f$ point-wise almost everywhere,  we obtain that  

$$
\int_{T_r}^{T_{r+1}}\int_{0}^{L}\int_{\mathbb{R}}f (\partial_t \varphi + v \partial_x \varphi )dvdxdt = \int_{T_r}^{T_{r+1}}\int_{0}^{L}\rho \partial_t \varphi + j \partial_x \varphi dxdt =  0.
$$

This shows the continuity equation. We then establish the regularity of $\rho$ and $j$ by using Theorem \ref{theoreme/regularite_trace} with $ (\phi_1,\phi_2) \in C_{0}^{\infty}(Q_{r})^2$.

\end{proof}

In the sequel, we will have to estimate the $L^1$ norm of non-negative functions in $BV$ knowing that they are transported by the flow of the vector field $(s,x,v) \longmapsto (1,v,E(s,x))$. To do so,  we apply the Green theorem inside the domain.

\begin{proposition}[$L^1$ estimate]
    Fix $t>0$ and  $E\in W^{1,\infty}(Q_t)$. Let $p$ be a nonnegative function such that  $p, |v|p\in BV(\Omega_t)$, and $Tp=0$ on $(0,t)\times\{L\}\times\mathbb{R^-}$. Then,  we have the following equality:
    
    \begin{equation}
        \|p(t)\|_{L^1(\mathcal{O})} \leq  \|p(0)\|_{L^1(\mathcal{O})}  + \int_{0}^t\int_{\mathbb{R}^+} v Tp(s,x,v)dsdv +  \int_{\Omega_t} \begin{pmatrix}
        1\\
        v\\
         E
    \end{pmatrix}\cdot d(Dp).
    \end{equation}

    \label{proposition/streamtube}
\end{proposition}

\begin{proof}
    Let $R>0$. Since $E \in \mathscr{C}^{0}\Big( \overline{Q_{t}}\Big)$ we may find a sequence $(E_n)_{n \in \mathbb{N}} \subset \mathscr{C}^{1}(Q_{t})$ which converges toward $E$ in $L^\infty(Q_t)$.  Applying the divergence formula \ref{divergence_formula} to the field \\ $\varphi : (s,x,v) \mapsto \chi_R(v)(1,v, E_n(s,x))$ and the scalar function $p$, we see that

\begin{equation*} 
    \int_{\Omega_t} p E_n(s,x)\chi_R'(v) dsdxdv = \int_{\partial\Omega_t} T(p) \chi_R(v)\begin{pmatrix}
        1\\
        v\\
         E_n(s,x)
    \end{pmatrix}\cdot\overrightarrow{n}d\mathcal{H}^2
    - \int_{\Omega_t} \chi_R(v)\begin{pmatrix}
        1\\
        v\\
         E_n(s,x)
    \end{pmatrix}\cdot d(Dp).
\end{equation*}

Using dominated convergence when $n\to\infty$, we get that 

\begin{eqnarray} \label{equation/3335}
    \int_{\Omega_t} p E(s,x)\chi_R'(v) dsdxdv &=& \int_{\partial\Omega_t} T(p) \chi_R(v)\begin{pmatrix}
        1\\
        v\\
         E(s,x)
    \end{pmatrix}\cdot\overrightarrow{n}d\sigma \nonumber
    \\ &-& \int_{\Omega_t} \chi_R(v)\begin{pmatrix}
        1\\
        v\\
         E(s,x)
    \end{pmatrix}\cdot d(Dp).
\end{eqnarray}

Then, we readily use the dominated convergence  as $R\to\infty$. Indeed, we know that $|v|T(p)\in L^1(\partial\Omega_t)$ (see Theorem \ref{theoreme/regularite_trace}). We thus obtain that

\begin{equation} \label{equation/3336}
   0 = \int_{\partial\Omega_t} T(p) \begin{pmatrix}
        1\\
        v\\
         E
    \end{pmatrix}\cdot\overrightarrow{n}d\sigma 
     - \int_{\Omega_t} \begin{pmatrix}
        1\\
        v\\
         E
    \end{pmatrix}\cdot d(Dp).
\end{equation}

Finally,  the boundary term is the sum of the boundary terms on each face of $\Omega_t$. Rearranging the terms and considering that $p,Tp\geq 0$ leads to the conclusion.

\end{proof}

\subsection{The upwind scheme}
Using the upwind scheme, we shall build a sequence of approximate solutions $(f_k)_{k\in\mathbb{N}}$ which converges toward the weak solution \\ $f \in C^{0}\Big([T_{r},T_{r+1}];L^{1}(\mathcal{O})\Big)$. Establishing BV estimates on this sequence and using the lower semi-continuity of the total variation in the $L^1_{loc}$ topology, we will deduce estimates of the weak solution in $BV(\mathcal{O})$. Let us introduce the scheme. We fix two integers $I,N\in\mathbb{N}^\star$ and set  the space step $h:=\dfrac{L}{I+1}$ and the time step $\Delta t: =\dfrac{T_{r+1}-T_r}{N}$. We define  the  following space points, velocity points and the times:
 
\begin{align}
    &x_i = \frac{h}{2}+ih, \ \  x_{i+1/2} = (i+1)h, \ \  x_{i-1/2}=ih, \  i=0,...,I\\
    &v_j = jh, \ \  v_{j+1/2} = jh+\frac{h}{2}, \ \  v_{j-1/2}=jh-\frac{h}{2}, \   j\in\mathbb{Z},\\
    &t_{n}= T_r +n\Delta t, \ n=0,...,N.
\end{align}
These points define the control volumes
\begin{equation}
     {K}_{n,i,j}= [t_n, t_{n+1})\times (x_{i-1/2},x_{i+1/2}) \times (v_{j-1/2},v_{j+1/2}), \ 0\leq n \leq N-1, \  0\leq i\leq I, \ \ j\in\mathbb{Z}.
\end{equation}
Their different faces (or cells) are denoted 
 \begin{align}
      \leftindex^n  {K}^j = [t_n, t_{n+1})\times (v_{j-1/2},v_{j+1/2}), \\
       K_i^j =  (x_{i-1/2},x_{i+1/2}) \times (v_{j-1/2},v_{j+1/2}),\\
      \leftindex^n  {K}_i = [t_n, t_{n+1})\times (x_{i-1/2},x_{i+1/2}).
 \end{align}
The data and the electric field are approximated by their mean value on each cell
\begin{align}
    \tilde{f}_{i,j} = \fint_{K_i^j} \tilde{f} dxdv, \ i=0,...I, j\in\mathbb{Z},\\
    g_{j}^{n} = \fint_{ \leftindex^n  {K}^j} g dtdv, \ n=0,...N-1, j\in\mathbb{N}^\star,\\
    E_i^n = \fint_{\leftindex^n  {K}_i}E dtdx, \  i=0,...I,  n=1,...,N-1.
\end{align}
For the sake of conciseness, we introduce the following set of indices:

\begin{align}
    \mathcal{N}(\mathcal{O}) := \Big \lbrace (i,j) \in \Iintv{0,I} \times \mathbb{Z} \: : \: \overline{K_{i}^{j}} \subset \mathcal{O} \Big \rbrace = \Iintv{1,I-1} \times \mathbb{Z},\\
    \mbox{(Incoming part)\ \  } \ \mathcal{N}^{-}(\partial{\mathcal{O}}):= \{0\}\times \mathbb{Z}^{+,\star} \cup \{I\}\times \mathbb{Z}^{-,\star},\\
    \mbox{(Outgoing part)\ \  } \ \mathcal{N}^{+}(\partial{\mathcal{O}}) := \{0\}\times \mathbb{Z}^- \cup \{I\}\times \mathbb{Z}^+.
\end{align}
The first set corresponds to interior cells while the two others sets correspond to cells whose frontier lies at the boundary. Note that $\mathcal{N}(\mathcal{O}) \cup \mathcal{N}^{+}(\partial{\mathcal{O}}) \cup \mathcal{N}^{-}(\partial{\mathcal{O}}) = \Iintv{0,I} \times \mathbb{Z}.$ A sketch of the mesh is given in Figure \ref{fig:enter-label}.
\begin{figure}[h!]
    \centering
    \includegraphics[width=0.4\linewidth]{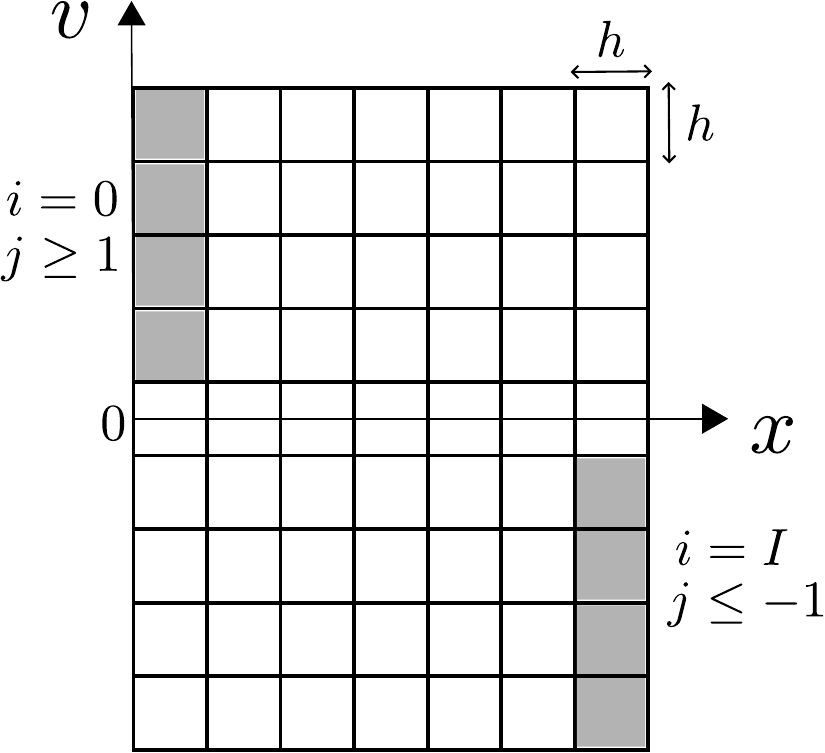}
    \caption{Partition of the phase plane in square cells of side $h$. Grey cells correspond to injection and absorption condition at $x=0$ and $x=1$ and are associated with the set of indices $\mathcal{N}^{-}(\partial \mathcal{O})$.}
    \label{fig:enter-label}
\end{figure}

We now define the sequence of approximate solutions we consider. For each values of $h$ and $\Delta t > 0$, we consider the Lipschitz continuous in time and piecewise constant in the phase-space function denoted $f_{h,\Delta t}$ defined for any $(t,x,v) \in \Omega_{r}$ by 

\begin{equation} \label{m_step_f_function}
f_{h,\Delta t} (t,x,v) = \sum_{i=0}^{I} \sum_{j\in\mathbb{Z}}  \Big(\frac{f_{i,j}^{n+1}(t-t_{n}) }{\Delta t} + \frac{f_{i,j}^{n}(t_{n+1}-t) }{\Delta t} \Big) \mathbb{1}_{K_{i,j}}(x,v), \quad t \in [t_{n},t_{n+1}),
\end{equation}
where the values $f^{n}_{i,j} \in \mathbb{R}$ are defined using the upwind scheme as follows. In $\mathcal{N}(\mathcal{O}) \cup  \mathcal{N}^{+}(\partial \mathcal{O})$ we define for $(n,i,j) \in \Iintv{0,N-1} \times \big( \mathcal{N}(\mathcal{O}) \cup  \mathcal{N}^{+}(\partial \mathcal{O}) \big),$

\begin{equation} \label{m_upwind_interior_cells}
\dfrac{f_{i,j}^{n+1}-f_{i,j}^{n}}{\Delta t}    +(v_j)^+ \dfrac{f_{i,j}^n-f_{i-1,j}^n}{h}   +(v_j)^- \dfrac{f_{i+1,j}^n-f_{i,j}^n}{h}    +\gamma (E_i^n)^+ \dfrac{f_{i,j}^n-f_{i,j-1}^n}{h}  +\gamma (E_i^n)^- \dfrac{f_{i,j+1}^n-f_{i,j}^n}{h} = 0,
\end{equation}
where $(\cdot)^{+} = \max\lbrace 0,\cdot\rbrace$ and $(\cdot)^{-} = \min\lbrace 0,\cdot\rbrace$ denote the positive and negative part functions.
At the incoming part of the boundary, we prescribed on the values for $(n,i,j) \in \Iintv{0,N-1} \times \mathcal{N}^{-}(\partial \mathcal{O}):$
\begin{align} 
& f^{n}_{0,j} =g_j^n, j \geq 1, \qquad f^{n}_{I,j} = 0 , j \leq -1 \label{m_upwind_bc}.
\end{align}
For the initial data, we define for $(i,j) \in \mathcal{N}(\mathcal{O}) \cup \mathcal{N}^{+}(\partial \mathcal{O})$,
\begin{equation} \label{m_upwind_initial_data}
    f_{i,j}^{0} = \tilde{f}_{i,j}.
\end{equation}
Note that in \eqref{m_upwind_bc} the incoming boundary conditions also appear at the initial step $n=0$. This choice appears more convenient for the estimates.
It is a known fact that the upwind scheme is positive and stable if $h$ and $\Delta t$ obey a CFL type condition. Hence, we set 
\begin{equation}
    \mu := \dfrac{\Delta t}{h}.
\end{equation}
The scheme equivalently writes for $(n,i,j) \in \Iintv{0,N-1} \times \big( \mathcal{N}(\mathcal{O}) \cup  \mathcal{N}^{+}(\partial \mathcal{O}) \big),$

\begin{align} \label{m_convex_comb_f}
   &f_{i,j}^{n+1}  = (1-\mu (v_j)^+ + \mu (v_j)^- - \mu \gamma (E_i^n)^+ + \mu \gamma (E_i^n)^-   ) f_{i,j}^n -   \mu (v_j)^- f_{i+1,j}^n  +   \mu (v_j)^+ f_{i-1,j}^n  \\
        &+ \mu \gamma (E_i^n)^+ f_{i,j-1}^n -  \mu \gamma (E_i^n)^- f_{i,j+1}^n. \nonumber
\end{align}
Observe that $f^{n+1}_{i,j}$ is a convex combination of the values  $f^n_{i\pm 1,j},f^n_{i,j\pm 1}$ if the data have compact support in velocity and if $\mu$ is small enough. This implies that the scheme preserves the positivity of the data and their $L^{\infty}$ bounds as we prove here after.
\begin{lemme}[Maximum principle]
    Assume that $T_{r+1}-T_r<1$ and that  $\tilde{f} \in L^{\infty}(\mathcal{O})$ and $g \in L^{\infty}\big( (T_{r},T_{r+1}) \times \mathbb{R}^{+} \big)$ vanish as soon as $|v| > A$ for some $A \geq 0.$ Then, provided the following the CFL condition 
    \begin{equation}
        \mu \leq \mu_0 := \dfrac{1-(T_{r+1}-T_{r})}{A+3/2+|\gamma| \| E \|_{L^{\infty}(Q_r)}}
    \label{equation/CFL}    
    \end{equation}
    holds, the scheme satisfies the maximum principle:
    \begin{eqnarray} \label{m_discrete_maximum_principle}
        0 \leq  f^{n}_{i,j} \leq \max \big\lbrace \underset{j \geq 1} \sup \:  g^{n}_{j} ; \underset{(i,j) \in \mathcal{N}(\mathcal{O}) \cup \mathcal{N}^{+}(\partial{ \mathcal{O}})} \sup \tilde{f}_{i,j}  \big \rbrace
    \end{eqnarray}  
    for all  $(n,i,j) \in \Iintv{0, N} \times \Iintv{0,I} \times \mathbb{Z}.$
\end{lemme}

\begin{proof}  Since the scheme is linear, we only prove the positivity as the $L^{\infty}$ the upper bound can be deduced easily by replacing $f^{n}_{i,j}$ by $ m - f_{i,j}^{n}$ with $m$ being the upper bound in \eqref{m_discrete_maximum_principle}. The proof is by induction. Observe that the boundary data are nonnegative so $f^{n}_{i,j}$ is nonnegative for $(i,j) \in \mathcal{N}^{-}(\partial \mathcal{O})$ thanks to \eqref{m_upwind_bc}. As for the other values.  The mean value $f^{0}_{i,j}$ vanishes if  $v_{j-1/2}> A$ and $j\geq 0$ or if $v_{j+1/2}< -A$ and $j\leq 0$, in other words $f^{0}_{i,j}=0$ if 

    $$
    j> \dfrac{A}{h}+\dfrac{1}{2}, \mbox{ or } j< -\dfrac{A}{h}-\dfrac{1}{2}.
    $$
    We set $j_{max} = \left\lfloor \dfrac{A}{h}+\dfrac{1}{2}\right\rfloor +1$. After one iteration of \eqref{m_convex_comb_f}, the support in velocity of the solution is extended by at most one cell, hence by a trivial induction, the support of  the sequence $f^n_{i,j}$ after $n$ iterations is contained in the set of indices $j$ such that $|j|\leq j_{max} + n$. By bounding $j_{max}$, we discover that the support of $f_{i,j}^n$ is confined to the $j$'s such that $\vert j \vert \leq \frac{A}{h}+\frac{3}{2}+n$.  Since $nh = \frac{t_n-T_{r}}{\mu}$, inside the support of $f_{i,j}^n$, we have $|v_j|\leq  A+\frac{3}{2}h+\frac{T_{r+1}-T_r}{\mu}$. To get a convex combination in \eqref{m_convex_comb_f}, we thus look for $\mu \geq 0$ such that

    $$
    1-\mu (v_j)^+ + \mu (v_j)^- - \mu \gamma (E_i^n)^+ + \mu \gamma (E_i^n)^- \geq 0. 
    $$
It suffices that $1\geq \mu(|v_j| + |\gamma|| E_{i}^n |)$ for all  $n\leq N$. By the bound on $|v_j|$ and the fact that the electric field is bounded, it suffices that
    
    $$
    1\geq \mu ( A+\frac{3}{2}h+\frac{T_{r+1}-T_r}{\mu} + |\gamma| \|E\|_{L^{\infty}})
    $$
    By rearranging terms, we find that the previous inequality is equivalent to
    $$
    \mu\leq  \dfrac{1-(T_{r+1}-T_{r})}{A+3/2h+|\gamma| \| E \|_{L^{\infty}(Q_r)}}.
    $$

    Since $0 \leq h \leq 1$, a sufficient condition for the positivity is the following:

    $$
    \mu \leq \mu_0 := \dfrac{1-(T_{r+1}-T_{r})}{A+3/2+|\gamma| \| E \|_{L^{\infty}(Q_r)}}\ \textnormal{ and } \ T_{r+1}-T_r\leq 1.
    $$
\end{proof}

The scheme written under the form \eqref{m_upwind_bc}-\eqref{m_upwind_initial_data} is not convenient to deduce $BV$ estimates when we perform discrete integration by parts. We therefore rewrite it a last time, by extending the solution $f_{h,\Delta t}$ to fictive cells.  The extension is given below 

\begin{equation}
    x_{-1}=-\frac{h}{2} \ ; \  \  f_{-1,j}^n = f_{0,j}^n, j \geq 0, \ n \in \Iintv{0,N}.
\end{equation}
\begin{equation}
    x_{I+1}=1+\frac{h}{2} \ ; \  f_{I+1,j}^n =f_{I,j}^n, \ j \leq 0, \ n \in \Iintv{0,N}.
\end{equation}
We now define a correction term  $e_{i,j}^n$  chosen so that \eqref{m_upwind_interior_cells} is written in $\Iintv{0,I} \times \mathbb{Z}$. We thus define for all $(n,i,j) \in \Iintv{0, N-1} \times \Iintv{0,I}\times \mathbb{Z} $,
\begin{eqnarray}
\label{m_upwind_corrected}
    f_{i,j}^{n+1}  & = & (1-\mu (v_j)^+ + \mu (v_j)^- - \mu \gamma (E_i^n)^+ + \mu \gamma (E_i^n)^-   ) f_{i,j}^n -   \mu (v_j)^- f_{i+1,j}^n  +   \mu (v_j)^+ f_{i-1,j}^n \nonumber \\
    &+&  \mu \gamma (E_i^n)^+ f_{i,j-1}^n -  \mu \gamma (E_i^n)^- f_{i,j+1}^n + e_{i,j}^n. 
\end{eqnarray}
with 
\begin{equation}
  e_{i,j}^n= 0, \quad  (n,i,j) \in \Iintv{0,N-1} \times  \mathcal{N}(\mathcal{O}) \cup \mathcal{N}^{+}(\partial \mathcal{O}) .\label{m_corrective_term_inside}
\end{equation} 
By writing the equality in the complementary part of $\mathcal{N}(\mathcal{O}) \cup \mathcal{N}^{+}(\partial \mathcal{O})$ and taking into account the boundary conditions and the extensions, we obtain for $(n,i,j) \in \Iintv{0,N-1} \times \mathcal{N}^{-}(\partial \mathcal{O})$

\begin{equation}
    e_{0,j}^n =  g_j^{n+1}-(1 - \mu \gamma (E_0^n)^+ + \mu \gamma (E_0^n)^-   ) g_{j}^n  - \mu \gamma (E_0^n)^+ f_{0,j-1}^n +  \mu \gamma (E_0^n)^- g_{j+1}^n, \quad j\geq 1, \label{m_corrective_terms_bc_1} 
\end{equation}

\begin{equation}
    e_{I,j}^{n} = 0,\quad j \leq -2, \label{m_corrective_terms_bc_2}\\
\end{equation} 

\begin{equation}
     e_{I,-1}^n = \mu\gamma (E_I^n)^- f_{I,0}^n. \label{m_corrective_terms_bc_3}
\end{equation}
So, the scheme \eqref{m_upwind_interior_cells}-\eqref{m_upwind_initial_data} is equivalent to \eqref{m_upwind_initial_data}, \eqref{m_upwind_corrected}-\eqref{m_corrective_terms_bc_3}.

\subsection{Discrete BV-estimates}
We shall prove several $BV$ estimates that are uniform in $h$ and $\Delta t$. To do so, we introduce
$$
v_h(v) = \sum_{j\in\mathbb{Z}} v_j \mathbb{1}_{(v_{j-1/2},v_{j+1/2})}(v).
$$
With a slight abuse in the notations, we define finite differences  $d_t,d_x,d_v$ for $f_{h,\Delta t}$ as follows:

\begin{equation}
    d_t f_{i,j}^n := f_{i,j}^{n+1}-f_{i,j}^n; \  d_x f_{i,j}^n := f_{i+1,j}^{n}-f_{i,j}^n; \  d_v f_{i,j}^n := f_{i,j+1}^{n}-f_{i,j}^n;
\end{equation}
Since $f_{h,\Delta t}$ is piecewise constant in space, it has an explicit total variation and $L^1$ norm in $\mathcal{O}$ for $t \in [t^{n},t^{n+1}):$
\begin{equation}
       \forall k \in \mathbb{N}, \quad  \Big\| |v_h|^kf_{h,\Delta t}(t) \Big \|_{L^{1}(\mathcal{O})}= \sum_{i=0}^{I} \sum_{j\in\mathbb{Z}} h^{2}|v_j|^k\Big( f_{i,j}^{n+1} \frac{t-t_{n}}{\Delta t} + f_{i,j}^{n} \frac{t_{n+1}-t}{\Delta t} \Big).
\end{equation}
The total variation in $\mathcal{O}$ is given for $t \in [t_{n},t_{n+1})$ by
$$TV(f_{h,\Delta t}(t),\mathcal{O}) = TV_{x}(f_{h,\Delta t}(t),\mathcal{O}) +TV_{v}(f_{h,\Delta t}(t),\mathcal{O})$$
where 

\begin{equation}
   TV_{x}(f_{h,\Delta t}(t),\mathcal{O}) :=  \sum_{i=0}^{I-1} \sum_{j\in\mathbb{Z}} h \Big( |d_x f_{i,j}^{n+1}| \frac{t-t_{n}}{\Delta t} +|d_x f_{i,j}^{n}| \frac{t_{n+1}-t}{\Delta t} \Big) ,
\end{equation}

\begin{equation}
    TV_{v}(f_{h,\Delta t}(t),\mathcal{O}) :=  \sum_{i=0}^{I} \sum_{j\in\mathbb{Z}} h  \Big( |d_v f_{i,j}^{n+1}| \frac{t-t_{n}}{\Delta t} +|d_v f_{i,j}^{n}| \frac{t_{n+1}-t}{\Delta t} \Big).
\end{equation}
As we may consider data that are not compactly supported in velocity, we need also to control the moment in velocity of $f_{h,\Delta t}$ in $BV(\mathcal{O}).$ We thus extend the preceding notation, and define for all $p\in\mathbb{N}$ and $t \in [t_{n},t_{n+1}):$

\begin{equation}
    TV_{x}^p(f_{h,\Delta t}(t),\mathcal{O}) :=    \sum_{i=0}^{I-1} \sum_{j\in\mathbb{Z}} h \Big( |v_j^pd_x f_{i,j}^{n+1}|\frac{t-t_{n}}{\Delta t} + |v_j^pd_x f_{i,j}^{n}|\frac{t_{n+1}-t}{\Delta t} \Big),
\end{equation}
\begin{equation}
    TV_{v}^p(f_{h,\Delta t}(t),\mathcal{O}) :=    \sum_{i=0}^{I} \sum_{j\in\mathbb{Z}} h  \Big( |v_j^pd_v f_{i,j}^{n+1}|\frac{t-t_{n}}{\Delta t} + |v_j^pd_v f_{i,j}^{n}|\frac{t_{n+1}-t}{\Delta t} \Big).
\end{equation}

The following lemma gives elementary bounds on the discretization of the data and are found by using Green formula on adjacent cells.

\begin{lemme}[Consistency estimates on the data] \label{m_consistency_data_estimates}
For all $p\in\mathbb{N}^{\star}$, there exists positive numbers  $a_0,...,a_{p-1}$ such that: 
    \begin{equation}
     \forall (i,j) \in \Iintv{0,I} \times \mathbb{Z}, \quad   |v_j|^p \tilde{f}_{i,j} \leq    h \sum_{k=0}^{p-1} a_k \fint_{K_{i,j}}v^k  \tilde{f}dxdv + \fint_{K_{i,j}}v^p  \tilde{f} dxdv, 
   \end{equation}

    \begin{equation}
\forall (n,j) \in \Iintv{0,N} \times \mathbb{Z}^{+,\star}, \quad       |v_j|^p g^n_{j} \leq   h \sum_{k=0}^{p-1} a_k \fint_{K_{n,j}}v^k  gdtdv + \fint_{K_{n,j}}v^p  gdtdv,
   \end{equation}

    \begin{multline}
     \forall (i,j) \in \Iintv{0,I} \times \mathbb{Z}, \quad  |v_j^p d_v\tilde{f}_{i,j}| \leq  h \sum_{k=0}^{p-1} a_k \left( \fint_{K_{i,j+1}}|v|^k  \tilde{f}dxdv +  \fint_{K_{i,j}}|v|^k \tilde{f}dxdv\right) \\
       + \dfrac{1}{h} |\partial_v(v^p\tilde{f})|(]x_{i-1/2},x_{i+1/2}[\times]v_{j-1/2},v_{j+3/2}[),
   \end{multline}

    \begin{equation}
\forall (i,j) \in \Iintv{0,I-1} \times \mathbb{Z}, \quad |v_j^p d_x\tilde{f}_{i,j}| \leq    \sum_{k=0}^{p} a_k h^{p-k-1}  |\partial_x(v^k  \tilde{f})|(]x_{i-1/2},x_{i+3/2}[\times ]v_{j-1/2},v_{j+1/2}[),
   \end{equation}

    \begin{multline} 
     \forall (n,j) \in \Iintv{0,N} \times \mathbb{Z}^{+,\star}, \quad  |v_j^p d_vg_{j}^n| \leq  h \sum_{k=0}^{p-1} a_k \left( \fint_{K_{n,j+1}}|v|^k  g dtdv +  \fint_{K_{n,j}}|v|^k g dtdv\right)     \\ + \dfrac{1}{\Delta t} (|\partial_v(v^pg)|(]t_n,t_{n+1}[\times]v_{j-1/2},v_{j+3/2}[),
       \label{m_estimate_vp_dvg}
   \end{multline}

    \begin{equation}
     \forall (n,j) \in \Iintv{0,N-1} \times \mathbb{Z}^{+,\star}, \quad |v_j^p d_tg_{j}^n| \leq  \sum_{k=0}^{p} a_k h^{p-k-1}  | \partial_t( |v|^k g)|(]t_{n},t_{n+2}[\times ]v_{j-1/2},v_{j+1/2}[).
   \end{equation}

   \begin{multline}
 \forall j\in\mathbb{Z}^+, \  |v_j|^p |\tilde{f}_{1,j}-g_j^0| \leq \sum_{k=0}^p a_k h^{p-k-1} (|\partial_x(|v|^k\tilde{f})|(K_{0,j}\cup K_{1,j}) + |\partial_t(|v|^kg)|(\leftindex^0 K_{j})  \\+ \|v^k (T\tilde{f}(0,.)-Tg(0,.))\|_{L^1(v_{j-1/2},v_{j+1/2})}) 
  + \| v^k T\tilde{f}(0,.)\|_{L^1(v_{j-1/2},v_{j+1/2})})
   \end{multline}

\label{lemme/estimationschema}
\end{lemme}

\begin{proof}
We have
   \begin{eqnarray*}
       v_j^p \tilde{f}_{i,j} &=&  \fint_{K_{i,j}}(v_j-v + v)^p \tilde{f}dxdv 
        =  \sum_{k=0}^{p} \binom{p}{k} \fint_{K_{i,j}}v^k (v_j-v)^{p-k} \tilde{f}dxdv.
   \end{eqnarray*}
So,
    \begin{eqnarray*}
       |v_j|^p \tilde{f}_{i,j} &\leq &   \sum_{k=0}^{p} \binom{p}{k} h^{p-k} \fint_{K_{i,j}}v^k  \tilde{f}dxdv \leq h \sum_{k=0}^{p-1} \binom{p}{k} h^{p-1-k} \fint_{K_{i,j}}v^k  \tilde{f}dxdv + \fint_{K_{i,j}}v^p  \tilde{f}dxdv.
   \end{eqnarray*}
The same computation holds for $v_j^pg_j^n$. Using the same idea we have
    \begin{eqnarray*}
       v_j^p d_v\tilde{f}_{i,j} &=&  \sum_{k=0}^{p-1} \binom{p}{k} \fint_{K_{i,j+1}}v^k (v_j-v)^{p-k} \tilde{f}dxdv -  \fint_{K_{i,j}}v^k (v_j-v)^{p-k} \tilde{f}dxdv \\
       &+ &\fint_{K_{i,j+1}} v^p \tilde{f} dxdv - \fint_{ K_{i,j}} v^p \tilde{f} dxdv.
   \end{eqnarray*}

We then observe that

\begin{eqnarray*}
    \fint_{K_{i,j+1}} v^p \tilde{f} dxdv - \fint_{ K_{i,j}} v^p \tilde{f} dxdv & = & \dfrac{1}{h} \int_{x_{i-1/2}}^{x_{i+1/2}} \int_{v_{j-1/2}}^{v_{j+3/2}} v^p\tilde{f} \partial_v\left(\dfrac{|v-v_{j+1/2}|-h}{h}\right) dxdv.
\end{eqnarray*}

Using the Green formula we obtain,
\begin{eqnarray*}
    \fint_{K_{i,j+1}} v^p \tilde{f} dxdv - \fint_{ K_{i,j}} v^p \tilde{f} dxdv & = & -\dfrac{1}{h} \int_{x_{i-1/2}}^{x_{i+1/2}} \int_{v_{j-1/2}}^{v_{j+3/2}}  \left(\dfrac{|v-v_{j+1/2}|-h}{h}\right) d(\partial_v(v^p\tilde{f})).
\end{eqnarray*}

Taking the absolute value we get,

\begin{eqnarray*}
    \left|\fint_{K_{i,j+1}} v^p \tilde{f} dxdv - \fint_{ K_{i,j}} v^p \tilde{f} dxdv \right|& \leq & \dfrac{1}{h} |\partial_v(v^p\tilde{f})|(]x_{i-1/2},x_{i+1/2}[\times]v_{j-1/2},v_{j+3/2}[).
\end{eqnarray*}

and then we conclude
    \begin{eqnarray*}
       |v_j^p d_v\tilde{f}_{i,j}| &\leq & h \sum_{k=0}^{p-1} \binom{p}{k} \left(\left(\frac{3h}{2}\right)^{p-k-1} + \left(\frac{h}{2}\right)^{p-k-1} \right)  \fint_{K_{i,j+1}}|v|^k  \tilde{f}dxdv +  \fint_{K_{i,j}}|v|^k \tilde{f}dxdv \\
       &+ &\dfrac{1}{h} |\partial_v(v^p\tilde{f})|(]x_{i-1/2},x_{i+1/2}[\times]v_{j-1/2},v_{j+3/2}[).
   \end{eqnarray*}
The proof is similar for $v_j^pd_v g_j^n$.
We reproduce the same computation for the finite difference in space $d_x$. We have

    \begin{eqnarray*}
       v_j^p d_x\tilde{f}_{i,j} &=&  \sum_{k=0}^{p} \binom{p}{k} \fint_{K_{i+1,j}}v^k (v_j-v)^{p-k} \tilde{f}dxdv -  \fint_{K_{i,j}}v^k (v_j-v)^{p-k} \tilde{f}dxdv \\
       & = &  \dfrac{1}{h}\sum_{k=0}^{p} \binom{p}{k} \int_{x_{i-1/2}}^{x_{i+3/2}}\int_{v_{j-1/2}}^{v_{j+1/2}}v^k (v_j-v)^{p-k} \tilde{f}  \partial_x\left(\dfrac{|x-x_{i+1/2}|-h}{h}\right)  dxdv \\
        & = &  -\dfrac{1}{h}\sum_{k=0}^{p} \binom{p}{k} \int_{x_{i-1/2}}^{x_{i+3/2}}\int_{v_{j-1/2}}^{v_{j+1/2}} (v_j-v)^{p-k} \left(\dfrac{|x-x_{i+1/2}|-h}{h}\right)  d(v^k \partial_x \tilde{f}).
   \end{eqnarray*}
Eventually taking the absolute value, we get the expected result. The computation is the same for the term $v_j^pd_t g_j^n$.  Finally, we shall prove the last point and control the difference of the data in a corner. We first expand their difference as 

$$
v_j^p (\tilde{f}_{1,j}-g_j^0) = v_j^p (\tilde{f}_{1,j}-\frac{1}{h}\int_{v_{j-1/2}}^{v_{j+1/2}}T\tilde{f}dv) + v_j^p \frac{1}{h}\int_{v_{j-1/2}}^{v_{j+1/2}}T\tilde{f}-Tg dv + v_j^p (\frac{1}{h}\int_{v_{j-1/2}}^{v_{j+1/2}}Tgdv-g_j^0).
$$
We bound the  first term, as the other one are tackled similarly. As $v_j^p = \sum_{k=0}^p a_k v^k (v_j-v)^{p-k}$, we expand the first term  as

\begin{multline}
 v_j^p (\tilde{f}_{1,j}-\frac{1}{h}\int_{v_{j-1/2}}^{v_{j+1/2}}T\tilde{f}dv) = \sum_{k=0}^p a_k   (\int_{K_{1,j} \cup K_{0,j}} \tilde{f} v^k (v_j-v)^{p-k} dxdv-\frac{1}{h}\int_{v_{j-1/2}}^{v_{j+1/2}} v^k (v_j-v)^{p-k}T\tilde{f}dv) 
 \\  +  \sum_{k=0}^p a_k   (-\int_{K_{0,j}} \tilde{f} v^k (v_j-v)^{p-k} dxdv+\frac{1}{h}\int_{v_{j-1/2}}^{v_{j+1/2}} v^k (v_j-v)^{p-k}T\tilde{f}dv) 
 \\  +   \sum_{k=0}^p a_k  -\frac{1}{h}\int_{v_{j-1/2}}^{v_{j+1/2}} v^k (v_j-v)^{p-k}T\tilde{f}dv. 
\end{multline}

Now, we apply Green's formula for BV functions to  $v^k (v_j-v)^{p-k}\tilde{f}$ and with the field $(x-2h,0)$:

\begin{multline}   \int_{K_{1,j}\cup K_{0,j}} \tilde{f} v^k (v_j-v)^{p-k} dxdv =  -   \int_{K_{1,j}\cup K_{0,j}} \tilde{f} v^k (v_j-v)^{p-k} (x-2h) d(\partial_x \tilde{f}) \\+ 2h\int_{v_{j-1/2}}^{v_{j+1/2}} v^k (v_j-v)^{p-k}T\tilde{f}dv.
\end{multline}

We thus get the desired bound 

$$
\left| \fint_{K_{1,j}\cup K_{0,j}} \tilde{f} v^k (v_j-v)^{p-k} dxdv - \fint_{v_{j-1/2}}^{v_{j+1/2}} v^k (v_j-v)^{p-k}T\tilde{f}dv \right| \leq h^{p-k-1} |\partial_x(|v|^k \tilde{f})|(K_{1,j}\cup K_{0,j}).
$$

In the same fashion, we have that 

$$
\left| \fint_{ K_{0,j}} \tilde{f} v^k (v_j-v)^{p-k} dxdv - \fint_{v_{j-1/2}}^{v_{j+1/2}} v^k (v_j-v)^{p-k}T\tilde{f}dv \right| \leq h^{p-k-1} |\partial_x(|v|^k \tilde{f})|( K_{0,j}).
$$

Finally, it is easy to see that 

$$
\left|   -\frac{1}{h}\int_{v_{j-1/2}}^{v_{j+1/2}} v^k (v_j-v)^{p-k}T\tilde{f}dv \right| \leq h^{p-k-1} \| v^k T\tilde{f}\|_{L^1(v_{j-1/2},v_{j+1/2})}.
$$

The conclusion follows.

\end{proof}

Before proving the main BV estimates on the approximation $f_{h,\Delta t}$, we need first to estimate the correction terms in terms of the data.

\begin{lemme}[BV-estimates of the corrective terms] \label{m_BV_estimate_corrective_terms}
Let $p \in \mathbb{N}$ and $0 \leq T_{r} < T_{r+1}\leq T$ with $T_{r+1} - T_{r} < 1$. Assume on the data that:
    \begin{itemize}
        \item[a)] for all $k \in \Iintv{0,p}$, $|v|^{k} \tilde{f}\in BV(\mathcal{O}) \cap L^\infty(\mathcal{O})$, $|v|^{k} g\in BV((T_{r},T_{r+1})\times\mathbb{R}^+) \cap L^\infty((T_{r},T_{r+1})\times\mathbb{R}^+)$
        \item[b)] $\underset{ t \in (T_{r},T_{r+1})} \sup \underset{ \Delta t > 0} \sup \dfrac{1}{\Delta t}|\partial_v(|v|^kg)|((t,t+\Delta t)\times\mathbb{R}^{+\star})<\infty$ for $k=0,...,p$;
        \item[c)] $\underset{ t \in (T_{r},T_{r+1})} \sup \underset{ \Delta t > 0} \sup \dfrac{1}{\Delta t}|\partial_t(|v|^kg)|((t,t+\Delta t)\times\mathbb{R}^{+\star})<\infty$ for $k=0,...,p$;
        \item[d)] $\underset{ t \in (T_{r},T_{r+1})} \sup  \int_{\mathbb{R}^{+}} |v|^k g(t,v)dv <\infty$ for $k=0,...,p$.
        \item[e)]  $\tilde{f}$ and $g$ vanish as soon as $|v| > A$ for some $A > 0.$
    \end{itemize}

Then, for all integers $m \leq N$

$$
\max_{0\leq n \leq m}\sum_{i=0}^{I}\sum_{j\in\mathbb{Z}} h |v_j^p d_ve_{i,j}^n|
+
\max_{0\leq n \leq m}\sum_{i=0}^{I-1}\sum_{j\in\mathbb{Z}} h |v_j^p d_xe_{i,j}^n | +  \max_{0\leq n \leq m}\sum_{i=0}^{I}\sum_{j\in\mathbb{Z}} h^{2} |v_j^p e_{i,j}^n| \underset{ \Delta t \rightarrow 0}=O(\Delta t).
$$
    \label{lemme/termeerreur}
\end{lemme}

\begin{proof} Let $ p \in \mathbb{N}$ and fix an integer $m \leq N.$
We detail the estimate of the first term. For $0 \leq n \leq m$, we  use the definition of $e_{i,j}^{n}$ \eqref{m_corrective_term_inside}-\eqref{m_corrective_terms_bc_3} to obtain

\begin{eqnarray*}
\sum_{i=0}^{I} \sum_{j \in \mathbb{Z}} h \big | v_{j}^{p} d_{v} e^{n}_{i,j} \big| &=& \sum_{i=0}^{I} \sum_{j \in \mathbb{Z}} h | v_{j}^{p} (e^{n}_{i,j+1}-e^{n}_{i,j})| \\
& = & \sum_{j=0}^{+\infty} h \big| v_{j}^{p}(e_{0,j+1}^{n}-e_{0,j}^{n}) \big| + \sum_{j=-1}^{-\infty} h \big| v_{j}^{p}( e_{I,j+1}^{n}-e_{I,j}^{n}) \big|,
\end{eqnarray*}

\noindent since the corrective term vanishes in $\mathcal{N}(\mathcal{O}) \cup \mathcal{N}^{+}(\partial \mathcal{O})$. Let us estimate the first term. By definition, we have 

\begin{equation*}
   R_{1} := \sum_{j=1}^{+\infty} h \Big |v_{j}^{p} (e_{0,j+1}^{n} - e_{0,j}^{n}) \Big| + h |v_{0}^{p} e_{0,1}^{n} |.
\end{equation*}

For $j \geq 1$, we have

\begin{equation*}
    e_{0,j+1}^{n}-e_{0,j}^{n} = d_{t}g_{j+1}^{n}-d_t g_j^n -(-\mu \gamma (E_0^n)^+ + \mu \gamma (E_0^n)^-) d_{v} g_{j}^{n} - \mu \gamma (E_{0}^{n})^{+} d_{v} f_{0,j-1}^{n} + \mu\gamma (E_{0}^{n})^{-} d_{v} g_{j+1}^{n},
\end{equation*}

\noindent and note that

 $$
d_{v}f_{0,j-1}^{n} = \begin{cases}
d_{v} g^{n}_{j-1} & \textnormal{ if } j \geq 2\\
g_{1}^{n} -f_{0,0}^{n} & \textnormal{ if } j = 1.
\end{cases}
$$

The term $R_1$  therefore satisfies

\begin{eqnarray*}
    R_{1} &\leq & h \Big( \sum_{j=1}^{+\infty} \Big | v_{j}^{p} d_{t}g_{j}^{n} \Big | + \Big| v_{j}^{p} d_{t} g_{j+1}^{n} \Big | \Big) + h \mu |\gamma| \| E \|_{L^{\infty}(Q_{r})}  \Big( \sum_{j=1}^{+\infty} \Big |v_{j}^{p} d_{v}g_{j}^{n} \Big | + \Big |v_{j}^{p} d_{v}g_{j+1}^{n} \Big |\Big)\\
   &+& h \mu | \gamma| \| E \|_{L^{\infty}(Q_{r})} \sum_{j=2}^{+\infty} \Big | v_{j}^{p} d_{v} g_{j-1}^{n}\Big| +  h \mu | \gamma| \| E \|_{L^{\infty}(Q_{r})} \big( |g_{1}^{n}| + |f_{0,0}^{n}| \big)\\
   &+& h|v_0|^p ( g_1^{n+1} + g_1^n (1+\mu h + \mu \gamma \|E\|_{L^\infty}(Q_r))  +  \mu \gamma \|E\|_{L^\infty}(f_{0,0}^n+g_2^n) + \mu h g_1^n ).
\end{eqnarray*}

It remains to estimate the second term

\begin{equation*}
    R_{2} := \sum_{j=-1}^{-\infty} h \big| v_{j}^{p}( e_{I,j+1}^{n}-e_{I,j}^{n}) \big|.
\end{equation*}

The definition of $e_{i,j}^n$  implies that

\begin{equation*}
    R_{2} = h\big | v_{-1}^{p} e^{n}_{I,-1} \big| + h\big| v_{-2}^{p} e^{n}_{I,-1} \big| \leq (1+2^{p}) h^{p+1} \mu |\gamma| \| E \|_{L^{\infty}(Q_{r})} | f_{I,0}^{n} |.
\end{equation*}

Summing the estimates on $R_{1}$ and $R_{2}$ we eventually obtain

\begin{eqnarray*}
\sum_{i=0}^{I}\sum_{j\in\mathbb{Z}} h |v_j^p d_ve_{i,j}^n| & \leq & h \Big( \sum_{j=1}^{+\infty} \Big | v_{j}^{p} d_{t}g_{j}^{n} \Big | + \Big| v_{j}^{p} d_{t} g_{j+1}^{n} \Big | \Big) \\ 
   &+& h \mu |\gamma| \| E \|_{L^{\infty}(Q_{r})}  \Big( \sum_{j=1}^{+\infty} \Big |v_{j}^{p} d_{v}g_{j}^{n} \Big | + \Big |v_{j}^{p} d_{v}g_{j+1}^{n} \Big |\Big)\\
   &+& h \mu | \gamma| \| E \|_{L^{\infty}(Q_{r})} \sum_{j=2}^{+\infty} \Big | v_{j}^{p} d_{v} g_{j-1}^{n}\Big| \\
   &+& h|v_0|^p ( g_1^{n+1} + g_1^n (1+\mu h + \mu \gamma \|E\|_{L^\infty(Q_r)})) \\
   & +& h|v_0|^p (\mu \gamma \|E\|_{L^\infty}(f_{0,0}^n+g_2^n) + \mu h g_1^n )\\
   &+& h \mu | \gamma| \| E \|_{L^{\infty}(Q_{r})} \big( |g_{1}^{n}|+ |f_{0,0}^{n}| \big) +  (1+2^{p}) h^{p+1} \mu |\gamma| \| E \|_{L^{\infty}(Q_{r})} | f_{I,0}^{n} |.
\end{eqnarray*}

Since $h\mu = \Delta t$ (recall that $\mu=\frac{\Delta t}{h}$), using the discrete maximum principle \eqref{m_discrete_maximum_principle} and the estimates of Lemma \ref{lemme/estimationschema}, all the terms except the two first lines  are readily $O(\Delta t)$ terms. For the two first lines, we remark that it is  $O(\Delta t)$ term thanks to the estimate \eqref{m_estimate_vp_dvg} and the assumptions b), c) and d).
Proceeding similarly, we treat the second term. We have

\begin{equation*}
    \sum_{i=0}^{I-1} \sum_{j \in \mathbb{Z}} h | v_{j}^{p} d_{x} e^{n}_{i,j} | = \sum_{j \in \mathbb{Z}} h(| v_{j}^{p} e^{n}_{I,j} | +  | v_{j}^{p} e^{n}_{0,j} |) = h| v_{-1}^{p} e^{n}_{I,-1} | + \sum_{j=1}^{+\infty} h | v_{j}^{p} e^{n}_{0,j}|.
\end{equation*}

We rearrange the terms in the expression of $e_{0,j}^{n}.$ One has

\begin{equation*}
    e^{n}_{0,j} =  g_j^{n+1}-(1 - \mu \gamma (E_0^n)^+ + \mu \gamma (E_0^n)^-   ) g_{j}^n  - \mu \gamma (E_0^n)^+ f_{0,j-1}^n +  \mu \gamma (E_0^n)^- g_{j+1}^n,
\end{equation*}

and we recall that 

$$
f_{0,j-1}^{n} = \begin{cases}
    f^{n}_{0,0} & \textnormal{ if } j =1,\\
    g^{n}_{j-1} & \textnormal{ if } j \geq 2.
\end{cases}
$$

It yields for $j \geq 2$

$$
e_{0,j}^{n} = d_{t}g^{n}_{j} + \mu \gamma (E_{0}^{n})^{+} d_{v}g_{j-1}^{n} + \mu \gamma (E_{0}^{n})^{-} d_{v}g_{j}^{n},
$$

\noindent and  for $j =1 $

$$
e_{0,1}^{n} = d_{t}g^{n}_{1} + \mu \gamma (E_{0}^{n})^{+} (g_1^n - f_{0,0}^n) + \mu \gamma (E_{0}^{n})^{-} d_{v}g_{1}^{n}.
$$

Using the expression of $e_{I,-1}^{n}$ given by \eqref{m_corrective_terms_bc_3} we eventually obtain
 
\begin{eqnarray*}
    \sum_{i=0}^{I-1}\sum_{j\in\mathbb{Z}} h  |v_j^p d_xe_{i,j}^n| & \leq & \sum_{j=1}^{+\infty} h |v_j|^p (|d_tg_j^n| + \mu |\gamma| \| E\|_{L^{\infty}(Q_r)} |d_vg_{j}^n| ) \\
    &+&\sum_{j=2}^{+\infty} h |v_{j}|^{p} \mu |\gamma| \| E\|_{L^{\infty}(Q_r)} |d_vg_{j-1}^n|   +  h^{p+1}\mu  |\gamma| \| E\|_{L^{\infty}(Q_r)} |f_{0,0}^n |\\
    &+& h^{p+1}\mu  |\gamma| \| E\|_{L^{\infty}(Q_r)} |f_{I,0}^n | + h^{p+1} \mu |\gamma|  \| E\|_{L^{\infty}(Q_r)} | g_{1}^{n} |.
\end{eqnarray*}

This time, assumption c) allows us to conclude that this sum is in $O(\Delta t)$.

Proceeding similarly, we also obtain

\begin{eqnarray*}
    \sum_{i=0}^{I}\sum_{j\in\mathbb{Z}} h^{2} |v_j^p e_{i,j}^n|& \leq & h^2 |v_1^p e_{0,1}^n| + h^2 |v_{-1}^p e_{I,-1}^n| + \sum_{j=2}^{+\infty}h^2|v_j^p e_{0,j}^n|.
\end{eqnarray*}

This terms already have been shown to be in $O(\Delta t)$.
\end{proof}

The next lemma provides us with the BV estimates we need to establish compactness of our sequence of approximate solution $f_{h,\Delta t}$

\begin{lemme}[BV-estimates on the approximate solution] \label{m_BV_estimate_main_terms} Let $ p \in \mathbb{N}$ and consider the assumptions of Lemma \ref{m_BV_estimate_corrective_terms}. Assume moreover that 
$$
\quad|v|^{p+1} g\in BV((T_r,T_{r+1})\times \mathbb{R}^{+\star}).
$$
Provided the CFL  condition (\ref{equation/CFL}) holds, for each $k = 0,...,p$, the map 
$$
t \in [T_{r},T_{r+1}] \longmapsto TV\Big(|v_h|^kf_{h,\Delta t}(t),\mathcal{O}\Big)+ \||v_h|^k f_{h,\Delta t}(t)\|_{L^1(\mathcal{O})}
$$
is bounded uniformly in $0 < h \leq 1$ and $0 < \Delta t \leq \mu_{0}h$ on $[T_{r},T_{r+1}].$

\label{lemme/variationtotale}
\end{lemme}

\begin{proof}
    Let $p \in \mathbb{N}$. We  use here the abuse of notation for $j \in \mathbb{Z}$, $v_{j \pm 1} = v_{j} \pm h $ which is either $v_{j+1}$ of $v_{j-1}$ depending on whether we use the $+$ sign of the $-$ sign. We will use in the practical computation two elementary inequalities. The first one is obtained by convexity of the function $s \in \mathbb{R} \longmapsto |s|^{p}:$
    
    \begin{equation}
         \forall j \in \mathbb{Z}, \quad    |v_{j\pm 1}|^p \leq 2^{p} (|v_j|^p + h^p). 
            \label{equation/observation1}
    \end{equation}
    
    The second one can be deduced by integration of the function $s \in \mathbb{R} \mapsto p \textnormal{ sgn}(s) |s|^{p-1}$ combined with the previous inequality \eqref{equation/observation1}:
    
    \begin{equation}
         \forall j \in \mathbb{Z}, \quad    | |v_{j\pm 1}|^p - |v_j|^p | \leq p h (|v_j|^{p-1} + |v_{j\pm 1}|^{p-1}) \leq ph(2^{p-1}+1) |v_{j}|^{p-1} + 2^{p-1} p h^p .
              \label{equation/observation2}
    \end{equation}
    
To estimate the total variations of $f_{h,\Delta t}$, we will use the discrete differentiation rules of a product of two sequences of real numbers $(a_{i,j})_{(i,j) \in \Iintv{-1,I+1} \times\mathbb{Z}}$ and $(b_{i,j})_{(i,j) \in \Iintv{-1,I+1} \times\mathbb{Z}}:$

\begin{align*}
d_{v}(a_{i,j}b_{i,j}) = b_{i,j+1}d_{v}(a_{i,j}) + a_{i,j} d_v(b_{i,j}), \quad (i,j) \in \Iintv{0,I} \times \mathbb{Z},\\
d_{x}(a_{i,j}b_{i,j})  = b_{i+1,j} d_x(a_{i,j}) + a_{i,j} d_x(b_{i,j}), \quad (i,j) \in \Iintv{0,I} \times \mathbb{Z}.
\end{align*}

We now begin with an estimate of  $TV_{v}^p(f_{h,\Delta t}(t), \mathcal{O}).$
Applying the operator $d_v$ to the scheme \eqref{m_upwind_corrected}, we find for $(i,j) \in \Iintv{0,I} \times \mathbb{Z},$

    \begin{eqnarray*}
        d_v(f_{i,j}^{n+1}) & = & (1-\mu v_j^+ + \mu v_j^- - \mu \gamma (E_i^n)^+ + \mu \gamma (E_i^n)^- ) d_v(f_{i,j}^n) \\
         & + & \mu v_j^+  d_v(f_{i-1,j}^n) - \mu v_j^-  d_v(f_{i+1,j}^n) + \mu \gamma (E_i^n)^+  d_v(f_{i,j-1}^n) - \mu \gamma (E_i^n)^-  d_v(f_{i,j+1}^n) \\
         & - & \mu d_v(v_j^+) d_x(f_{i-1,j+1}^n)  - \mu d_v(v_j^-) d_x(f_{i,j+1}^n) + d_v(e_{i,j}^n).
    \end{eqnarray*}
    
We thus obtain after multiplication by $v_j^p$ that

 \begin{eqnarray*}
        v_j^pd_v(f_{i,j}^{n+1}) & = & (1-\mu v_j^+ + \mu v_j^- - \mu \gamma (E_i^n)^+ + \mu \gamma (E_i^n)^- ) v_j^pd_v(f_{i,j}^n) \\
         & + & \mu v_j^+ v_j^p d_v(f_{i-1,j}^n) - \mu v_j^- v_j^p  d_v(f_{i+1,j}^n) + \mu \gamma (E_i^n)^+ v_j^p d_v(f_{i,j-1}^n) \\
         &-& \mu \gamma (E_i^n)^- v_j^p d_v(f_{i,j+1}^n)   -  \mu d_v(v_j^+) v_j^p d_x(f_{i-1,j+1}^n)  - \mu d_v(v_j^-) v_j^p d_x(f_{i,j+1}^n) \\
         & +& v_j^p d_v(e_{i,j}^n).
    \end{eqnarray*}  
    
Under the CFL condition (\ref{equation/CFL}), we have $(1-\mu v_j^+ + \mu v_j^- - \mu \gamma (E_i^n)^+ + \mu \gamma (E_i^n)^- ) \geq 0$.
So we can take the module, use the triangle inequality, sum on the spatial index $0\leq i\leq I$ and rearrange the terms. We eventually obtain

\begin{eqnarray*}
    \sum_{i=0}^{I} |v_j^p  d_v(f_{i,j}^{n+1}) | & \leq &  \sum_{i=0}^{I} |v_j^p  d_v(f_{i,j}^{n}) | + \mu v_j^+ ( |v_j^p  d_v(f_{-1,j}^{n}) | - |v_j^p  d_v(f_{I,j}^{n}) |)   \\
    &+& \mu v_j^- ( | v_j^p d_v(f_{0,j}^{n}) | - |v_j^p  d_v(f_{I+1,j}^{n}) |) \\
     & + & \mu |\gamma| \sum_{i=0}^{I} |(E_i^n)^-| ( |v_j^p  d_v(f_{i,j}^{n}) | - |  v_j^pd_v(f_{i,j+1}^{n}) | ) \\
     &+&  \mu |\gamma| \sum_{i=0}^{I} |(E_i^n)^+| ( | v_j^p d_v(f_{i,j-1}^{n}) | - |v_j^p  d_v(f_{i,j}^{n}) | ) \\
    & + & \mu |d_v(v_j^+)| \sum_{i=0}^{I}|v_j^pd_x(f_{i-1,j+1}^n)|  + \mu |d_v(v_j^-)| \sum_{i=0}^{I}|v_j^pd_x(f_{i,j+1}^n)| + \sum_{i=0}^{I} |v_j^pd_v(e_{i,j}^n)|.
\end{eqnarray*}

Nonpositive boundary terms are less than $0$, so

\begin{eqnarray*}
    \sum_{i=0}^{I} |v_j^p  d_v(f_{i,j}^{n+1}) | & \leq &  \sum_{i=0}^{I} |v_j^p  d_v(f_{i,j}^{n}) | + \mu v_j^+  |v_j^p  d_v(f_{-1,j}^{n}) | - \mu v_j^-   |v_j^p  d_v(f_{I+1,j}^{n}) | \\
     & + & \mu |\gamma| \sum_{i=0}^{I}|(E_i^n)^-| ( |v_j^p  d_v(f_{i,j}^{n}) | - |  v_j^pd_v(f_{i,j+1}^{n}) | ) \\
     &+&  \mu |\gamma| \sum_{i=0}^{I} |(E_i^n)^+| ( | v_j^p d_v(f_{i,j-1}^{n}) | - |v_j^p  d_v(f_{i,j}^{n}) | ) \\
    & + & \mu |d_v(v_j^+)| \sum_{i=0}^{I}|v_j^pd_x(f_{i-1,j+1}^n)|  + \mu |d_v(v_j^-)| \sum_{i=0}^{I}|v_j^pd_x(f_{i,j+1}^n)| + \sum_{i=0}^{I} |v_j^pd_v(e_{i,j}^n)|.
\end{eqnarray*}

Then, we sum over $j\in\mathbb{Z}$ to get

\begin{eqnarray*}
    \sum_{j\in\mathbb{Z}}\sum_{i=0}^{I} |v_j^p  d_v(f_{i,j}^{n+1}) | & \leq &    \sum_{j\in\mathbb{Z}}\sum_{i=0}^{I} |v_j^p  d_v(f_{i,j}^{n}) | + \mu   \sum_{j=1}^{+\infty}  |v_j^{p+1}  d_v(f_{-1,j}^{n}) | + \mu   \sum_{j=-\infty}^{-1}   |v_j^{p+1}  d_v(f_{I+1,j}^{n}) | \\
     & + & \mu |\gamma| \sum_{i=0}^{I}  \sum_{j\in\mathbb{Z}} |(E_i^n)^-| ( |v_j^p | - |  v_{j-1}^p| )  | d_v(f_{i,j}^{n})| \\
     &+&  \mu |\gamma| \sum_{i=0}^{I}\sum_{j\in\mathbb{Z}}  |(E_i^n)^+| ( |v_{j+1}^p | - |  v_{j}^p| )  | d_v(f_{i,j}^{n})| \\
    & + & \mu h \sum_{j=1}^{+\infty}  \sum_{i=-1}^{I-1}|v_{j-1}^pd_x(f_{i,j}^n)|  +  \mu h \sum_{j=-1}^{-\infty}  \sum_{i=0}^{I}|v_{j-1}^pd_x(f_{i,j}^n)| + \sum_{j\in\mathbb{Z}} \sum_{i=0}^{I} |v_j^pd_v(e_{i,j}^n)|.
\end{eqnarray*}

Notice that $d_x(f_{-1,j}^n)=0$ if $j\geq0, n\geq 0$ and that $d_x(f_{I,j}^n)=0$ if $j\leq0, n\geq 0$. By using observations (\ref{equation/observation1})-(\ref{equation/observation2}), we get that

\begin{eqnarray*}
    \sum_{j\in\mathbb{Z}}\sum_{i=0}^{I} |v_j^p  d_v(f_{i,j}^{n+1}) | & \leq &    \sum_{j\in\mathbb{Z}}\sum_{i=0}^{I} |v_j^p  d_v(f_{i,j}^{n}) | + \mu   \sum_{j=1}^{+\infty}   |v_j^{p+1}  d_v(f_{-1,j}^{n}) | + \mu   \sum_{j=-\infty}^{-1}   |v_j^{p+1}  d_v(f_{I+1,j}^{n}) | \\
     & + & hp (2^{p-1}+1) \|E\|_{L^{\infty}} \mu |\gamma| \sum_{i=0}^{I}  \sum_{j\in\mathbb{Z}}   | v_j^{p-1}d_v(f_{i,j}^{n})|\\
     &+&  h^p p 2^{p-1} \|E\|_{L^{\infty}} \mu |\gamma| \sum_{i=0}^{I}  \sum_{j\in\mathbb{Z}}   |d_v(f_{i,j}^{n})|\\
    & + & 2^p \mu h \sum_{j\in\mathbb{Z}}  \sum_{i=0}^{I-1}|v_{j}^p d_x(f_{i,j}^n)|  + 2^p \mu h^{p+1} \sum_{j\in\mathbb{Z}}  \sum_{i=0}^{I-1}|d_x(f_{i,j}^n)|  \\
    &+& \sum_{j\in\mathbb{Z}} \sum_{i=0}^{I} |v_j^pd_v(e_{i,j}^n)|.
\end{eqnarray*}

So for any integer $0 \leq k \leq n+1$, we obtain

\begin{eqnarray*}
    \max_{0\leq n\leq k+1}\sum_{j\in\mathbb{Z}}\sum_{i=0}^{I} h|v_j^p  d_v(f_{i,j}^{n}) | & \leq &    \max_{0\leq n\leq k}\sum_{j\in\mathbb{Z}}\sum_{i=0}^{I}h |v_j^p  d_v(f_{i,j}^{n}) | + \mu   \max_{0\leq n\leq k}\sum_{j=1}^{+\infty} h  |v_j^{p+1}  d_v(f_{-1,j}^{n}) | \\
    &+& \mu   \max_{0\leq n\leq k}\sum_{j=-\infty}^{-1}h   |v_j^{p+1}  d_v(f_{I+1,j}^{n}) | \\
     & + & hp (2^{p-1}+1) \|E\|_{L^{\infty}} \mu |\gamma| \max_{0\leq n\leq k}\sum_{i=0}^{I}  \sum_{j\in\mathbb{Z}} h  | v_j^{p-1}d_v(f_{i,j}^{n})|\\
     &+&  h^p p 2^{p-1} \|E\|_{L^{\infty}} \mu |\gamma| \max_{0\leq n\leq k}\sum_{i=0}^{I}  \sum_{j\in\mathbb{Z}}h  |d_v(f_{i,j}^{n})|\\
    & + & 2^p \mu h \max_{0\leq n\leq k}\sum_{j\in\mathbb{Z}}  \sum_{i=0}^{I-1}h|v_{j}^p d_x(f_{i,j}^n)| \\
    & +&v 2^p \mu h^{p+1} \max_{0\leq n\leq k}\sum_{j\in\mathbb{Z}}  \sum_{i=0}^{I-1}h|d_x(f_{i,j}^n)|  \\
    &+&\max_{0\leq n\leq k} \sum_{j\in\mathbb{Z}} \sum_{i=0}^{I} h |v_j^pd_v(e_{i,j}^n)|.
\end{eqnarray*}

We then move on to $ TV_x^{p}(f_{h,\Delta t}, \mathcal{O})$. Applying  $d_x$ to the scheme \eqref{m_upwind_corrected}, then multiplying by $v_j^p$ and passing to the absolute value,  we get that

\begin{eqnarray*}
    |v_j^pd_x(f_{i,j}^{n+1})| & \leq & |v_j^pd_x(f_{i,j}^{n})|  + \mu v_{j}^+ (|v_j^pd_x(f_{i-1,j}^n)|-|v_j^pd_x(f_{i,j}^n)|) \\
    &+& \mu v_{j}^- (|v_j^pd_x(f_{i,j}^n)|-|v_j^pd_x(f_{i+1,j}^n)|) \\
     & + & \mu|\gamma| |(E_i^n)^+| (|v_j^pd_x(f_{i,j-1}^n)|-|v_j^pd_x(f_{i,j})|)\\
     & +& \mu|\gamma| |(E_i^n)^-| (|v_j^pd_x(f_{i,j}^n)|-|v_j^pd_x(f_{i,j+1}^n)|) \\
      & + & \mu|\gamma| |d_x((E_i^n)^+)| |v_j^pd_v(f_{i+1,j-1}^n)| + \mu|\gamma| |d_x((E_i^n)^-)| |v_j^pd_v(f_{i+1,j}^n)| + |v_j^pd_x(e_{i,j}^n)|.
\end{eqnarray*}

Summing over $0\leq i \leq I-1 $, and bounding nonpositive terms by $0$, we get

\begin{eqnarray*}
     \sum_{i=0}^{I-1}|v_j^pd_x(f_{i,j}^{n+1})| & \leq &  \sum_{i=0}^{I-1}|v_j^pd_x(f_{i,j}^{n})|  + \mu   |v_j^{p+1}d_x(f_{-1,j}^n)| + \mu  |v_j^{p+1}d_x(f_{I,j}^n)| \\
     & + & \mu|\gamma| \sum_{i=0}^{I-1} |(E_i^n)^+| (|v_j^pd_x(f_{i,j-1}^n)|-|v_j^pd_x(f_{i,j})|) \\
     & +&  \mu|\gamma| \sum_{i=0}^{I-1} |(E_i^n)^-| (|v_j^pd_x(f_{i,j}^n)|-|v_j^pd_x(f_{i,j+1}^n)|) \\
      & + & \mu|\gamma| \sum_{i=0}^{I-1}|d_x((E_i^n)^+)| |v_j^pd_v(f_{i+1,j-1}^n)| + \mu|\gamma| \sum_{i=0}^{I-1}|d_x((E_i^n)^-)| |v_j^pd_v(f_{i+1,j}^n)| \\
      &+& \sum_{i=0}^{I-1}|v_j^pd_x(e_{i,j}^n)|.
\end{eqnarray*}

Then we sum over $j\in\mathbb{Z}$ and we find 

\begin{eqnarray*}
     \sum_{j\in\mathbb{Z}}\sum_{i=0}^{I-1}|v_j^pd_x(f_{i,j}^{n+1})| & \leq &  \sum_{j\in\mathbb{Z}}\sum_{i=0}^{I-1}|v_j^pd_x(f_{i,j}^{n})|  + \mu  \sum_{j=1}^{+\infty} |v_j^{p+1}d_x(f_{-1,j}^n)| + \mu  \sum_{j=-\infty}^{-1} |v_j^{p+1}d_x(f_{I,j}^n)| \\
     & + & \mu|\gamma| \sum_{j\in\mathbb{Z}}\sum_{i=0}^{I-1} |(E_i^n)^+| (|v_{j+1}^p|-|v_j^p|)  |d_x(f_{i,j}^n)| \\
     & + & \mu|\gamma| \sum_{j\in\mathbb{Z}}\sum_{i=0}^{I-1} |(E_i^n)^-| (|v_{j}^p|-|v_{j-1}^p|)  |d_x(f_{i,j}^n)| \\
      & + & \mu|\gamma|  \sum_{j\in\mathbb{Z}}\sum_{i=0}^{I-1}|d_x((E_i^n)^+)| |v_{j+1}^pd_v(f_{i+1,j}^n)| \\
      &+& \mu|\gamma|\sum_{j\in\mathbb{Z}} \sum_{i=0}^{I-1}|d_x((E_i^n)^-)| |v_j^pd_v(f_{i+1,j}^n)| \\
      &+&  \sum_{j\in\mathbb{Z}}\sum_{i=0}^{I-1}|v_j^pd_x(e_{i,j}^n)|.
\end{eqnarray*}

As $E\in W^{1,\infty}(Q_r)$, we have $ |d_x((E_i^n)^+)|+|d_x((E_i^n)^-)| \leq \| \partial_x E \|_{L^\infty(Q_r)} h$. We also use the estimates (\ref{equation/observation1}, \ref{equation/observation2}) which yields

\begin{eqnarray*}
     \sum_{j\in\mathbb{Z}}\sum_{i=0}^{I-1}|v_j^pd_x(f_{i,j}^{n+1})| & \leq &  \sum_{j\in\mathbb{Z}}\sum_{i=0}^{I-1}|v_j^pd_x(f_{i,j}^{n})|  + \mu  \sum_{j=1}^{+\infty} |v_j^{p+1}d_x(f_{-1,j}^n)| + \mu  \sum_{j=-\infty}^{-1} |v_j^{p+1}d_x(f_{I,j}^n)| \\
     & + & \mu|\gamma|\|E\|_{L^\infty(Q_r)} p (1+2^{p-1}) h  \sum_{j\in\mathbb{Z}}\sum_{i=0}^{I-1}  |v_j^{p-1}d_x(f_{i,j}^n)| \\
     & + &  \mu|\gamma|\|E\|_{L^\infty(Q_r)} p 2^{p-1} h^p  \sum_{j\in\mathbb{Z}}\sum_{i=0}^{I-1}  |d_x(f_{i,j}^n)| \\
     & + & 2^p\mu|\gamma|\|\partial_x E\|_{L^\infty(Q_r)} h \sum_{j\in\mathbb{Z}}\sum_{i=0}^{I}|v_j^p d_v(f_{i,j}^n)| \\
     & + & 2^p\mu|\gamma|\|\partial_x E\|_{L^\infty(Q_r)} h^{p+1} \sum_{j\in\mathbb{Z}}\sum_{i=0}^{I}| d_v(f_{i,j}^n)| \\
      &+&  \sum_{j\in\mathbb{Z}}\sum_{i=0}^{I-1}|v_j^pd_x(e_{i,j}^n)|.
\end{eqnarray*}
So for $k \in \Iintv{0,N-1}$ we eventually obtain,
\begin{eqnarray*}
     \max_{0\leq n \leq k+1}\sum_{j\in\mathbb{Z}}\sum_{i=0}^{I-1}h|v_j^pd_x(f_{i,j}^{n})| & \leq &  \max_{0\leq n \leq k}\sum_{j\in\mathbb{Z}}\sum_{i=0}^{I-1}h|v_j^pd_x(f_{i,j}^{n})|  + \mu  \max_{0\leq n \leq k}\sum_{j=1}^{+\infty}h |v_j^{p+1}d_x(f_{-1,j}^n)| \\
     &+ &\mu  \max_{0\leq n \leq k}\sum_{j=-\infty}^{-1}h |v_j^{p+1}d_x(f_{I,j}^n)| \\
     & + & \mu|\gamma|\|E\|_{L^\infty(Q_r)} p (1+2^{p-1}) h  \max_{0\leq n \leq k}\sum_{j\in\mathbb{Z}}\sum_{i=0}^{I-1}h  |v_j^{p-1}d_x(f_{i,j}^n)| \\
     & + &  \mu|\gamma|\|E\|_{L^\infty(Q_r)} p 2^{p-1} h^p  \max_{0\leq n \leq k}\sum_{j\in\mathbb{Z}}\sum_{i=0}^{I-1}h  |d_x(f_{i,j}^n)| \\
     & + & 2^p\mu|\gamma|\|\partial_x E\|_{L^\infty(Q_r)} h \max_{0\leq n \leq k}\sum_{j\in\mathbb{Z}}\sum_{i=0}^{I}h|v_j^p d_v(f_{i,j}^n)| \\
     & + & 2^p\mu|\gamma|\|\partial_x E\|_{L^\infty(Q_r)} h^{p+1} \max_{0\leq n \leq k}\sum_{j\in\mathbb{Z}}\sum_{i=0}^{I}h| d_v(f_{i,j}^n)| \\
      &+&  \max_{0\leq n \leq k}\sum_{j\in\mathbb{Z}}\sum_{i=0}^{I-1}h|v_j^pd_x(e_{i,j}^n)|.
\end{eqnarray*}

Set $V_k^p =  \underset{ 0 \leq n \leq k } \max \displaystyle \sum_{j\in\mathbb{Z}}\sum_{i=0}^{I} h|v_j^p  d_v(f_{i,j}^{n}) |  +  \underset{ 0 \leq k \leq n }   \max \sum_{j\in\mathbb{Z}}\sum_{i=0}^{I-1} h|v_j^p  d_x(f_{i,j}^{n}) | $.  Summing the two estimates on $TV_{v}^{p}(f_{h,\Delta t}(t), \mathcal{O})$ and $TV_{x}^{p}(f_{h,\Delta t}(t),\mathcal{O})$, we glean the induction for $k \in \Iintv{0,N-1}:$

\begin{eqnarray*}
    V_{k+1}^p  & \leq & (1+A_p\Delta t ) V_k^p + pB_p \Delta t V_k^{p-1} + \Delta t(pC_{p} h^{p-1}+ D_p h^p)V_k^0 + B_k^p,
\end{eqnarray*}

\noindent where $A_p,B_p,C_{p},D_p$ are  positive constants which depend only on $p$, $|\gamma|$, $\|E\|_{L^{\infty}(Q_{r})}$, $\| \partial_{x} E \|_{L^{\infty}(Q_{r})}$. The term $B_k^p \geq 0$ gathers all the sums corrective terms $e_{i,j}^n$, along with the two boundary terms 

$$
\mu   \max_{0\leq n\leq k}\sum_{j=1}^{+\infty} h  |v_j^{p+1}  d_v(f_{-1,j}^{n}) | + \mu   \max_{0\leq n\leq k}\sum_{j=-\infty}^{-1}h   |v_j^{p+1}  d_v(f_{I+1,j}^{n}) | .
$$

 These terms are controlled by the total variations of $|v|^kg$, for $k=0,..p+1$. A discrete Grönwall inequality then yields the explicit estimate: $\forall k\in  \Iintv{0,N}$,
 
$$
V_k^p \leq (1+A_p\Delta t)^kV_0^p + \left(\dfrac{(1+A_p\Delta t)^k-1}{A_p}\right)\dfrac{ pB_p \Delta t V_k^{p-1} + \Delta t(pC_{p} h^{p-1}+ D_p h^p)V_k^0 + B_k^p}{\Delta t}.
$$

 In addition, Lemma \ref{m_BV_estimate_corrective_terms} proves that $\underset{ k \in \Iintv{0,N}} \max B_k^p= O(\Delta t)$, hence the remainder is bounded for all $0 < \mu \leq \mu_{0}.$ Observe that $V_0^p$ is controlled by the $L^1$ norms \\ $\| v^k (T\tilde{f}(0,.)-Tg(T_{r},.))\|_{L^1(\mathbb{R}^+)}, \| v^k\tilde{f}\|_{L^1(\mathcal{O})}, \|v^k g\|_{L^1((T_{r},T_{r+1})\times\mathbb{R})}, \| v^k T\tilde{f}(0,.)\|_{L^1(\mathbb{R}^+)}$, \\ $\| v^k T\tilde{f}(L,.)\|_{L^1(\mathbb{R}^-)} $ and the total variations $TV\Big( |v|^k g,(T_{r},T_{r+1})\times\mathbb{R}^{+\star}\Big), TV\Big( |v|^k\tilde{f},\mathcal{O}\Big) $ for $k=0,...,p$. An induction on $p$ shows that $V_k^0,V_k^1,...,V_k^p$ are bounded uniformly with respect to $k$. In the same fashion, we derive the estimate for the discrete $L^{1}$ norms $I_{p}^{k} := \underset{ 0 \leq n \leq k} \max \displaystyle \sum_{j \in \mathbb{Z}} \sum_{i = 0}^{I} h^2 |v_{j}|^{p} |f^{n}_{i,j}|$:

$$
I_{k+1}^p \leq I_k^p + \Delta t \sum_{j=1}^\infty h|v_j|^{p+1}g_j^n + \Delta t (2^p+1) V_k^p + 2^p h^p \Delta t V_k^0 + \sum_{j\in\mathbb{Z}}\sum_{i = 0}^{I} h^2 |v_j|^p |e_{i,j}^n|. 
$$

The same reasoning applies. It shows the claim since the total variation and the $L^{1}$ norms on $\mathcal{O}$ for $t \in [t_{n},t_{n+1})$ are obtained by convex combination of the norms at step $n$ and $n+1.$
\end{proof}

\subsection{Passing to the limit in the scheme}

We are ready to apply a compactness argument to identify  a weak solution for the Vlasov problem \eqref{systeme/vlasovlineairebord} and prove Proposition \ref{proposition:regularite_vlasov}.

\begin{proof}[Proof of Proposition \ref{proposition:regularite_vlasov}] We first show the result when $\tilde{f},g$ have bounded supports, and then extend the result to general data. \\

\noindent \textbf{Step 1: compactly supported data.} \\

Let $T_{0} = 0$ and $0 < T_{1} < 1$. Fix $0 < \mu \leq \mu_{0}$, with $\mu_{0}$ is given in the CFL condition \eqref{equation/CFL}. Recall that $\Delta t = \mu h_{I}$ with $h_{I} = \frac{1}{I+1}$ and set for ease $(f_{I} = f_{h_{I},\mu h_{I}})_{I \in \mathbb{N}}.$ First, observe that the support in velocity of $f_I$ on $[T_{0},T_{1}]$ is bounded by some constant $R> 0$. We then work in the domain $\mathcal{O}_R :=[0,L]\times [-R,R]\subset \mathcal{O}$. By the convergence result of Boyer \cite{boyerconv} (see also  \cite{aguillon}), we know that $(f_I)_{I\in\mathbb{N^\star}} $ converges in the space $C^0([0,T_{1}]; L^1(\mathcal{O}_R))$ toward a function $f$ which is the unique weak solution  (Theorem \ref{m_def_weaksol_vlasov}). By Fatou's lemma,  lower semi-continuity of the total variation, and Lemma \ref{m_BV_estimate_main_terms}) we get for $k \in \Iintv{0,..p}$ and $t \in [0,T_{1}],$

\[
\||v|^kf(t)\|_{BV(\mathcal{O}_R)} \leq \underset{ l \rightarrow +\infty} \liminf \||v_{h_{I_{l}}}|^kf_{I_{l}}(t)\|_{BV(\mathcal{O}_R)} < +\infty.
\]

Recall that the discrete estimates do not depend on the size of the supports of $\tilde{f},g$, but rather on the $BV$, $L^1$ and $L^\infty$ norms of their moments in velocity. Hence, every bound is uniform in the size of the supports.  Eventually, using the equation  
 
 $$\partial_tf = -v\partial_xf-\gamma E \partial_vf,$$
 
 \noindent we obtain $\partial_{t} f \in \mathcal{M}(\Omega_{0})$  since $v\partial_{x}f \in \mathcal{M}(\Omega_{0})$ and $E \partial_{v} f \in \mathcal{M}(\Omega_{0}).$  So $f \in BV(\Omega_{0}).$ Multiplying the equation by $v^{k}$ for $0 \leq k \leq p-1$ and using the fact that the differentiation in time and the multiplication by $v$
commutes,  we obtain that $|v|^{k} f \in BV(\Omega_{0})$ for $0 \leq k \leq p-1.$
Repeating the process on  a finite number of adjacent time intervals of size less than one gives us the result on $\Omega_T$. We finish by proving that the  trace $\overline{T}f$ of the weak solution (see Theorem \ref{m_weak_form})  coincide with the usual trace (see Theorem \ref{divergence_formula} and  \ref{theoreme/trace}). To do this, we first consider a sequence $(E_n)_{n\in\mathbb{N}}$ such that $E_n\in C_0^\infty(\overline{\Omega_T})$ and $E_n\xrightarrow[n\to \infty]{L^\infty(Q_T)} E $. By performing an integration by parts  with $f$ and the vector field $\phi(1,v,\gamma E_n)$, where $\phi\in C^\infty_0(\overline{\Omega_T})$, and taking the limit $n\to\infty$, we get that 

\begin{equation}
\int_{\Omega_T} f (\partial_t\phi + v\partial_x\phi +\gamma E\partial_v \phi) dtdxdv = \int_{\partial\Omega_T} Tf \phi (1,v,\gamma E)\cdot \overrightarrow{n} d\mathcal{H}^2 - \int_{\Omega_T} \phi (1,v,\gamma E)\cdot d(Df).
\end{equation}

Now, by choosing $\phi\in C_0^\infty(\Omega_T)$ and by comparing this last equality with the weak formulation of the Vlasov equation, we deduce that $(1,v,\gamma E)\cdot Df = 0$. we therefore deduce that for  every $\phi\in C^\infty_0(\overline{\Omega_T})$, 

\begin{equation}
\int_{\Omega_T} f (\partial_t\phi + v\partial_x\phi +\gamma E\partial_v \phi) dtdxdv = \int_{\partial\Omega_T} Tf \phi (1,v,\gamma E)\cdot \overrightarrow{n} d\mathcal{H}^2. 
\end{equation}

Again, comparing this last equality with the weak formulation gives that $Tf = \overline{T}f$ almost everywhere on $\partial\Omega_T$. \\

\noindent \textbf{Step 2: extension to  data with unbounded supports.} \\

We  now get rid of the compactness assumption on the support of $g$ and $\tilde{f}$.  Choose  $R>0$, and consider the solutions of the following Vlasov problem:

\begin{equation}
\left\lbrace
\begin{array}{ll}
\partial_t f_R + v \partial_x f_R +\gamma E \partial_v f_R = 0, &  (t,x,v)\in  \Omega_T, \\ 
f_R(0,x,v) =  \chi_R(v)\tilde{f}(x,v), & (x,v)\in \mathcal{O},  \\
f_R(t,0,v) =  \chi_R(v)g(t,v),   &  (t,v)\in (0,T)\times\mathbb{R}^{+\star}, \\
f_R(t,L,v) = 0,   &  (t,v)\in (0,T)\times\mathbb{R}^{-\star}. 
\end{array}\right.
\label{systeme/vlasovbis}
\end{equation}
 The solution $f_R$ satisfies the same bounds as before which are uniform in $R$. Indeed for $k \in \Iintv{0,p},$ we have
$$
\|v^k\chi_R g\|_{L^1([0,T]\times\mathbb{R})} \leq \|v^k g \|_{L^1([0,T]\times\mathbb{R})},
$$
$$
\|v^k \chi_R g \|_{L^\infty([0,T]\times\mathbb{R})} \leq \|v^k g\|_{L^\infty([0,T]\times\mathbb{R})},
$$
$$
|D(\chi_R v^k g)|((0,T)\times\mathbb{R}^{+\star}) \leq 2T\|\alpha'\|_{L^1([0,1])} \| v^k g \|_{L^\infty(\Omega_T)} + |D(v^kg)|((0,T)\times\mathbb{R}^{+\star}),
$$

\noindent and so on. For each $t \in [0,T]$, the set

\begin{equation}
   \mathcal{S}_t= \Big \lbrace f_{R}(t)  \: : R>0 \Big \rbrace 
\end{equation}

\noindent is bounded in $BV(\mathcal{O})$, thus for every $t$ there is a subsequence $(f_{R_n}(t))_{n\in\mathbb{N}}$  which converges in $L^1_{loc}(\mathcal{O})$ toward a function $f(t)\in BV_{loc}(\mathcal{O})$ (Theorem \ref{theoreme/compacite}). Once again, By Fatou's lemma, lower semi-continuity of the total variation and  Lemma \ref{m_BV_estimate_main_terms}, we obtain that   for $k \in \Iintv{0,..p}$ and $t \in [0,T]$

\[
\||v|^kf(t)\|_{BV(\mathcal{O})} \leq \underset{ n \rightarrow +\infty} \liminf \||v|^kf_{R_n}(t)\|_{BV(\mathcal{O})} < +\infty,
\]  

\noindent hence $|v|^kf(t)\in BV(\mathcal{O})$. The regularity in $BV(\Omega_T)$ is deduced as before. We now show that $f$ satisfies the weak formulation. Consider $V$ a bounded  neighborhood of   $Supp(\phi)\cap \partial\Omega_T$ in $\mathbb{R}^3$ and consider $A = \overline{V\cap \Omega_T}$. The set $A$ is compact in $\mathbb{R}^3$ and its inverse image by the flow of $(1,v,\gamma E)$ is a compact of the incoming part of $\partial \Omega_T$. Hence, for all $n$ large enough, for every $(t,x,v)\in A$, and for all $r>0$ small enough,

$$
\fint_{B((t,x,v),r)} f_{R_n}(t',x',v')) dt'dx'dv' = \fint_{B((t,x,v),r)} f(t',x',v')) dt'dx'dv' .
$$

We deduce that $\lim\limits_{n\to\infty} Tf_{R_n} = Tf$ almost everywhere in $A$. Using this and dominated convergence,  we get the weak formulation on $f$ for every $\phi\in C_0^\infty(\overline{\Omega_T})$:

\begin{equation}
\int_{\Omega_T} f (\partial_t\phi + v\partial_x\phi +\gamma E\partial_v \phi) dtdxdv = \int_{\partial\Omega_T} Tf \phi (1,v,\gamma E)\cdot \overrightarrow{n} d\mathcal{H}^2 .
\end{equation}

It only remains to establish that $f\in C^0([0,T];L^1(\mathcal{O}))$. We  fix $B>0$,  two instants $t_1<t_2$ and define $U_{t_1,t_2}=(t_1,t_2)\times\mathcal{O}$. By a smoothing argument, we build a sequence of functions $(\lambda_n)_{n\in\mathbb{N}}$ such that for all $n$,  $\lambda_n\in C^\infty_c(\mathbb{R})$, $|\lambda_n|\leq 1$, $Supp(\lambda_n)\in [-B-1,B+1]$ and that $\lambda_n\to \mathbb{1}_{[-B,B]}$ almost everywhere in $\mathbb{R}$. We choose $\psi\in C^\infty_0(\overline{\mathcal{O}})$. Notice that $\psi$ is smooth up to the boundary. Applying Green's identity to the vector field $(\psi \lambda_n,0,0)$ and the  function $f\in BV(U_{t_1,t_2})$, we get that

$$
\int_{\mathcal{O}}(f(t_2)-f(t_1))\psi \lambda_n dxdv = \int_{U_{t_1,t_2}} \psi \lambda_n d(\partial_t f).
$$

Letting $n\to\infty$ and using the dominated convergence theorem, we obtain

$$
\forall \psi\in C^\infty([0,L]\times[-B,B]), \ \int_{\mathcal{O}}(f(t_2)-f(t_1))\psi  dxdv = \int_{U_{t_1,t_2}} \psi  d(\partial_t f).
$$

Now, we choose $\varphi\in L^\infty(\mathcal{O})$, with support in $(0,L)\times(-B,B)$. Using mollifier and imitating (\cite{evansgariepy}, Theorem 3, section 4.2) we built a sequence $(\varphi_k)_{k\in\mathbb{N}}$ such that, for all $k\in\mathbb{N}$,  $\varphi_k \in C^\infty([0,L]\times[-B,B])$, $|\varphi_k|\leq \|\varphi\|_{L^\infty([0,L]\times[-B,B])}$. Moreover, $\varphi_k \to \varphi$ almost everywhere in $[0,L]\times[-B,B]$.  Using the last equality for this sequence and using dominated convergence again, we find that 

$$
\forall \varphi\in L^\infty(\mathcal{O}) \mbox{ such that } Supp(\varphi)\subset (0,L)\times(-B,B), \ \int_{\mathcal{O}}(f(t_2)-f(t_1))\varphi  dxdv = \int_{U_{t_1,t_2}} \varphi  d(\partial_t f).
$$

Finally, if the support of  $\varphi$ is not bounded support, we multiply it by the cutoff function $\chi_R$, apply the last result to this product, let $R\to\infty$ and use dominated convergence, to get that 

$$
\forall \varphi\in L^\infty(\mathcal{O}), \  \int_{\mathcal{O}}(f(t_2)-f(t_1))\varphi  dxdv = \int_{U_{t_1,t_2}} \varphi  d(\partial_t f).
$$

By duality, we get the inequality $\|f(t_2)-f(t_1)\|_{L^1(\mathcal{O})} \leq |\partial_tf|(U_{t_1,t_2})$. This proves the time continuity in $L^1(\mathcal{O})$.

\end{proof}

\section{The Vlasov-Poisson problem with prescribed   Dirichlet boundary condition} \label{sec:Vlasov-Poisson}

In this section, we recall the proof of wellposedness of the Vlasov-Poisson  system when Dirichlet boundary conditions are prescribed (\ref{melon}, \ref{meleche}).  For a given   function $u\in W^{1,\infty}([0,T])$, we consider the two species Vlasov-Poisson problem  on $[0,T]$:

\begin{equation} 
\left\lbrace
\begin{array}{ll}
\partial_t f^{\pm} + v \partial_x f^{\pm}  -\gamma_\pm \partial_x\phi \partial_v f^{\pm}  = 0 &  \textnormal{ in } \Omega_{T}, \\ 
f^\pm(t,0,v) = g(t,v) &\text{ if } (t,v) \in (0,T) \times \mathbb{R}^{+\star}, \\
f^\pm(t,L,v) = 0 &\text{ if } (t,v) \in (0,T) \times \mathbb{R}^{-\star}, \\ 
f^\pm(0,x,v) = \tilde{f}(x,v) & \text{ if }  (x,v)\in \mathcal{O},  \\
\end{array}\right.
\label{melon}
\end{equation}

\begin{equation}
\left\lbrace
\begin{array}{lll}
-\varepsilon_0 \partial_{xx}\phi & = & q_+ \rho^+ + q_- \rho^- \textnormal{ in } [0,T)\times (0,L), \\
\phi(t,0) &=& 0, \ \forall t \in [0,T), \\ 
\phi(t,L)& =& u(t),\ \forall t\in [0,T).
\end{array}
\right.
\label{meleche}
\end{equation}

Existence and uniqueness in BV for this problem has already  been studied \cite{Guo}. Nevertheless, we point out that the dependence of the estimates with respect to $u$ is fundamental for the Vlasov-Poisson-Ampère system, so we reproduce them for completeness. The proof proceeds by  a fixed point theorem applied to the electric field. We consider the map 

\begin{equation*}
    \Lambda : W^{1,\infty}(Q_{T}) \longmapsto L^{\infty}(Q_{T})
\end{equation*}

\noindent defined for any $E\in W^{1,\infty}(Q_T)$  and $(t,x) \in Q_{T}$ by

\begin{eqnarray}
\Lambda(E)(t,x) &:=& -\dfrac{1}{L\epsilon_{0}}\int_0^L\int_0^s (q_+\rho^++q_-\rho^-)(t,u) duds + \dfrac{1}{\epsilon_{0}}\int_0^x(q_+\rho^++q_-\rho^-)(t,u) du\nonumber \\
& - &\frac{u(t)}{L},
\label{systeme/ptfixe1}
\end{eqnarray}

\noindent where $\rho^+ ,\rho^-$ are computed from the solutions of  Vlasov equation \eqref{systeme/vlasovlineairebord} ($T_r=0,T_{r+1}=T$) when the electric field $E$ is given. According to Proposition \ref{proposition:regularite_vlasov} for each given $E \in W^{1,\infty}(Q_{T})$ we have unique weak solution such that $|v|^{k} f^{+},|v|^{k} f^{-} \in BV(\Omega_{T})$ for $k = 0,...p-1$. Thus, it is easy to see that $\Lambda(E)$ defines a function in $L^{\infty}(Q_{T})$ since the total charge and $u$ are bounded. In the sequel, we shall apply the Schauder fixed point theorem to $\Lambda$ for the existence.  At the very first, we need to prove $\Lambda\Big(W^{1,\infty}(Q_{T}) \Big) \subset W^{1,\infty}(Q_{T})$. From the formula \eqref{systeme/ptfixe1}, it is easy to see that $\Lambda(E), \partial_{x} \Lambda(E) \in L^{\infty}(Q_{T}).$ To prove that $\partial_{t} \Lambda(E) \in L^{\infty}(Q_{T})$ we first need to compute this derivative. The next lemma provides an expression similar to the Maxwell-Ampère equation.
\begin{lemme}
Provided $|v|^{k} f^{\pm}$  belongs to $BV(\Omega_{T})$ for $k = 0,1$, we have the following expression for the distributional derivative $\partial_{t} \Lambda(E)$:
\begin{equation}
\epsilon_{0} \partial_t\Lambda(E)(t,x) = -(q_+j^++q_-j^-)(t,x) + \frac{1}{L} \int_0^L(q_+j^++q_-j^-)(t,u)du - \frac{\epsilon_{0} }{L}u'(t).
\end{equation}
\label{lemme/maxwellampere}
\end{lemme}
\begin{proof}

Since $|v|^{k} f^{\pm}$  belongs to $BV(\Omega_{T})$ for $k = 0,1$, the moments $\rho^\pm,j\pm$ belong to $BV(Q_T)$. Let us now compute the distributional derivative $\partial_t \int_0^x(q_+\rho^+ +q_- \rho^-)(t,s) ds$. Let $\varphi\in C_0^\infty(Q_T)$.  Performing an integration by parts, using the continuity equation, and doing again an integration by parts, we justify the following computation:

\begin{eqnarray*}
\int_{Q_T} \partial_t \int_0^x \rho^\pm(t,s) ds \varphi(t,x)dxdt & = & -\int_{Q_T} \int_0^x \rho^\pm(t,s) ds \partial_t\varphi(t,x)dxdt  \\
 & = & -\int_0^L \left( \int_0^x\int_0^T \rho^\pm(t,s)\partial_t\varphi(t,x) dsdt \right)dx \\
 & = &  \int_0^L\left( \int_0^x\int_0^T \varphi(t,x) d\partial_t\rho^\pm(t,s)\right) dx  \\
 & = & -\int_0^L \left(\int_0^x\int_0^T \varphi(t,x) d\partial_s j^\pm(t,s) \right) dx  \\
  & = & -\int_{Q_T} \varphi(t,x) (j^\pm(t,x)-j^\pm(t,0)) dxdt 
\end{eqnarray*}

Proceeding similarly for the other terms, we get the result.

\end{proof}

We now prove existence and uniqueness for the Vlasov-Poisson system.
\begin{theoreme}[Existence and uniqueness for the Vlasov-Poisson system] Assume, for an integer $p\geq 2$, that:
\begin{itemize}
        \item $|v|^k\tilde{f}^\pm \in  BV(\mathcal{O}), k=0,...p$,
        \item $|v|^k g^\pm\in BV((0,T)\times\mathbb{R}^+), k=0,...p+1$,
        \item $|v|^k \tilde{f}^\pm\in L^\infty(\mathcal{O}),|v|^k g^\pm\in L^\infty((0,T)\times\mathbb{R}^{+\star})$, $k=0,...p$,
          \item $\underset{ t \in (0,T)}{\sup} \underset{ \Delta t > 0}  \sup \dfrac{1}{\Delta t}|\partial_v(|v|^kg^{\pm})|((t,t+\Delta t)\times\mathbb{R}^{+\star})< + \infty, k=0,...p$, 
        \item $\underset{ t \in (0,T)}{\sup} \underset{ \Delta t > 0}  \sup \dfrac{1}{\Delta t}|\partial_t(|v|^kg^{\pm})|((t,t+\Delta t)\times\mathbb{R}^{+\star})< + \infty, k=0,...p$, 
               \item $\underset{ t \in (0,T)} \sup  \int_{\mathbb{R}^{+}} |v|^k g^\pm(t,v)dv <\infty$, $k=0,...,p$.
\end{itemize}

Then the Vlasov-Poisson system (\ref{melon}, \ref{meleche}) has a unique weak solution $(f^{+},f^{-},\phi) \in C^0\Big([0,T];L^{1}(\mathcal{O}) \Big) ^2  \times W^{2,\infty}(Q_{T}) $ such that $|v|^{k} f \in BV(\Omega_{T})$ for all $k \in \Iintv{0,p-1}$, and that for all  $k \in \Iintv{0,p}$ and almost all $t\in(0,T),$ $|v|^k f(t)\in BV(\mathcal{O})$.

\end{theoreme}

\begin{proof}

    The first part of the proof consists in finding $W^{1,\infty}$ estimates on $\Lambda(E)$. Therefore, we start by bounding $\Lambda(E)$.  We notice that for $t \in [0,T],$
    
    $$
    \left|\int_0^x (q_+\rho^++q_-\rho^-)(t,u)du \right| \leq q_+\|f^+(t)\|_{L^1(\mathcal{O})}+ |q_-|\|f^-(t)\|_{L^1(\mathcal{O})}.
    $$
    
For each species, we have the mass estimate, for all $t \in [0,T]$,

    $$\|f^{\pm}(t)\|_{L^1(\mathcal{O})} \leq \|\tilde{f}^{\pm}\|_{L^1(\mathcal{O})} + \| vg^{\pm} \|_{L^1((0,T) \times \mathbb{R}^{+})}.$$
    
So we obtain directly from \eqref{systeme/ptfixe1}

\begin{align}
    &\epsilon_{0} \|\Lambda(E)\|_{L^\infty(Q_T)}\leq \frac{\varepsilon_0}{L}\|u\|_{L^\infty((0,T))} \\ \label{m_estimate_E}
    &+ 2(q_++|q_-|)(\|\tilde{f}^+\|_{L^1(\mathcal{O})} + \|v g^+ \|_{L^1((0,T)\times\mathbb{R}^+)} + \|\tilde{f}^-\|_{L^1(\mathcal{O})} + \|v g^- \|_{L^1((0,T)\times\mathbb{R}^+)}). \nonumber
\end{align}

We then turn to $\partial_x \Lambda(E)$. As  $ \epsilon_{0} \partial_x \Lambda(E)=q_+\rho^++q_-\rho^-$, 

\begin{eqnarray*}
    \epsilon_{0} \|\partial_x\Lambda(E)\|_{L^\infty(Q_T)} &\leq & (q_+\|\rho^+\|_{L^\infty(Q_T)} +|q_-| \|\rho^-\|_{L^\infty(Q_T)} ) \\
    &\leq &(q_+ +|q_-|) \int_\mathbb{R}\dfrac{dv}{1+v^2}(\|(1+v^2)f^+\|_{L^\infty(\Omega_T)} + \|(1+v^2)f^-\|_{L^\infty(\Omega_T)} ).
\end{eqnarray*}

Using the expression of $\partial_t\Lambda(E)$, we get in the same vein

\begin{align*}
    &\epsilon_{0} \|\partial_t\Lambda(E)\|_{L^\infty(Q_T)} \leq \frac{\epsilon_{0}}{L}\|u'\|_{L^\infty((0,T))} \\
    &+ 2(q_++|q_-|) \int_\mathbb{R}\dfrac{dv}{1+v^2}(\|(|v|+|v|^3)f^+\|_{L^\infty(\Omega_T)} + \|(|v|+|v|^3)f^-\|_{L^\infty(\Omega_T)} ).
\end{align*}

Hence, we see that we need bounds on $\||v|^k f^\pm\|_{L^\infty(\Omega_T)}$ for $k=0,1,2,3$. To this end, we recall that for $(\tau,x,v) \in \Omega_{T}$, we denote by $t\in [0,T] \longmapsto  (X(t,\tau,x,v), V(t,\tau,x,v) )$ the characteristic curve associated with the field $(v,\gamma_{\pm} E)$ which starts from $(\tau,x,v)$. We define the three functions below:
    
    $$
    \Gamma_1(t)= V(t,\tau,x,v) f^\pm(t,X(t,\tau,x,v),V(t,\tau,x,v)),
    $$
    $$
    \Gamma_2(t)= V(t,\tau,x,v)^2 f^\pm(t,X(t,\tau,x,v),V(t,\tau,x,v)),
    $$
    $$
    \Gamma_3(t)= V(t,\tau,x,v)^3 f^\pm(t,X(t,\tau,x,v),V(t,\tau,x,v)).
    $$
    
    The function $f^\pm$ is known to be constant on the characteristic curves. It is given explicitly by
    
    $$
    f^{\pm}(t,x,v) = \mathbb{1}_{\tau(t,x,v) = 0} \tilde{f}\Big(X(0,t,x,v), V(0,t,x,v)\Big) +\mathbb{1}_{\tau(t,x,v) > 0} g^{\pm}\Big(\tau(t,x,v), V(\tau(t,x,v),t,x,v)\Big),
    $$
    
    \noindent where $\tau(t,x,v)$ denotes the time of exit of the trajectory $t \in [0,T] \longmapsto (t,X(t,x,v),V(t,x,v))$ from $\Omega_{T}$. We can therefore differentiate $\Gamma_1$ to get for $t \in (0,T)$,
    
    $$
    \Gamma'_1(t) = \gamma_\pm E(t,X(t,\tau,x,v))f^\pm(t,X(t,\tau,x,v),V(t,\tau,x,v)).
    $$
    
    Integrating the previous inequality on $[\tau,t]$ yields  $|\Gamma_1(t)|\leq |\Gamma_1(\tau)| + \int_\tau^t |\Gamma_1'(s)|ds$, so that we get  
    
   $$
 \|vf^\pm \|_{L^\infty(\Omega_T)} \leq \|v \tilde{f}^\pm\|_{L^{\infty}(\mathcal{O})}+ \|v g^\pm\|_{L^{\infty}((0,T)\times\mathbb{R})} +  T |\gamma_\pm| \|E\|_{L^\infty(Q_T)} \|f^\pm \|_{L^\infty(\Omega_T)}.
    $$ 
    
    Analogous computations lead us to the other bounds:
    
    $$
    \|v^2f^\pm \|_{L^\infty(\Omega_T)} \leq \|v^2 \tilde{f}^\pm\|_{L^{\infty}(\mathcal{O})}+ \|v^2 g^\pm\|_{L^{\infty}((0,T)\times\mathbb{R})} +  2T |\gamma_\pm| \|E\|_{L^\infty(Q_T)} \|vf^\pm \|_{L^\infty(\Omega_T)},
    $$
    $$
    \|v^3f^\pm \|_{L^\infty(\Omega_T)} \leq \|v^3 \tilde{f}^\pm\|_{L^{\infty}(\mathcal{O})}+ \|v^3 g^\pm\|_{L^{\infty}((0,T)\times\mathbb{R})} +  3T |\gamma_\pm |\|E\|_{L^\infty(Q_T)} \|v^2f^\pm \|_{L^\infty(\Omega_T)}.
    $$
    
    Moreover, the maximum principle states that 
    
    $$\|f^\pm \|_{L^\infty(\Omega_T)}\leq   \max(\|\tilde{f}^\pm\|_{L^\infty(\mathcal{O})},\|g^\pm\|_{L^\infty(Q_{T})}).$$
    
    So, there exists a polynomial $P(X)=AX^3+BX^2+CX+D$ with positive coefficients $A,B,C,D$ depending only on the data   such that 
    
    $$
     \|\partial_t\Lambda(E)\|_{L^\infty(Q_T)} +  \|\partial_x\Lambda(E)\|_{L^\infty(Q_T)} \leq P(T\|E\|_{L^\infty(Q_T)}) .
    $$
    
    We are ready to apply the Schauder fixed point theorem. We have just proven that the range of $\Lambda$ is a subset of 
    
    \begin{equation*}
     B_{M,T} =\left\lbrace
        \begin{array}{cc}
            E \in W^{1,\infty}(Q_T) \ : \ \|E\|_{L^\infty(Q_T)}\leq M ; \ \|\partial_t E\|_{L^\infty(Q_T)} +  \|\partial_x E\|_{L^\infty(Q_T)} \leq P(T M) 
        \end{array}\right\rbrace,
    \end{equation*}
    
\noindent  where $\epsilon_{0} M = \frac{\varepsilon_0}{L}\|u\|_{L^\infty((0,T))} + (q_++|q_-|)(\|\tilde{f}^+\|_{L^1(\mathcal{O})} + \|v g^+ \|_{L^1((0,T)\times\mathbb{R}^+)} + \|\tilde{f}^-\|_{L^1(\mathcal{O})} + \|v g^- \|_{L^1((0,T)\times\mathbb{R}^+)})$. Thus, we see that $\Lambda (B_{M,T}) \subset B_{M,T}$  where $B_{M,T}$ is closed convex subset of $L^{\infty}(Q_T)$ endowed with the norm $\|.\|_{L^{\infty}(Q_T)}$. We also see that $B_{M,T}$ is compact for the $L^{\infty}(Q_{T})$ topology. Indeed, any sequence of functions in $B_{M,T}$ satisfies the assumptions of Ascoli theorem, thus has a subsequential limit that converges for the norm $\|.\|_{L^\infty(Q_T)}$ to an element of $B_{M,T}$. So far, we have proven that $\Lambda : B_{M,T} \mapsto B_{M,T}$ and that $B_{M,T}$ is compact for the $\|.\|_{L^\infty(Q_T)}$ norm. We now prove that $\Lambda$ is Lipschitz continuous. To this end, consider two elements $E_1$, $E_2$ of $W^{1,\infty}(Q_T)$ and the associated densities $f_1^\pm, f_2^\pm$. Set $h^\pm = f_1^\pm-f_2^\pm$, and $\psi=E_1-E_2$, then $h^\pm \in BV(\Omega_{T})$ satisfies the following transport equation with source term 

\begin{equation}
\partial_t h^\pm + v \partial_x h^\pm +\gamma_+ E_1 \partial_v h^\pm = -\gamma_\pm \psi \partial_vf_2^\pm  \textnormal{ in } \mathcal{M}(\Omega_{T}) ,
\end{equation}

\noindent where the initial and boundary conditions are zero. By the chain rule theorem in $BV$ (Theorem \ref{theoreme/chainrule}), $|v^{k} h^\pm| \in  BV(\Omega_T)$ for $k=0,...,p-1$. Using an $L^1$ estimate along the field $(1,v,\gamma_\pm E_1)$ (Proposition \ref{proposition/streamtube}) we find that 

$$
\|h^\pm(t)\|_{L^1(\mathcal{O})} \leq \int_{\Omega_t} \begin{pmatrix}
    1 \\
    v \\
    \gamma_\pm E_1
\end{pmatrix}
\cdot d(D|h^\pm|).
$$

Using again the chain rule (Theorem \ref{m_chain_rule}), we have that $|h^{\pm}| \in BV(\Omega_{T})$ and moreover its derivative is a measure that decomposes as

$$
D|h^\pm| = \textnormal{sgn}(\overline{h}^\pm) \tilde{D}h^\pm + \dfrac{|(h^\pm)_+| - |(h^\pm)_-|}{h^\pm_+-h^\pm_-} D_{jump}h^\pm,
$$

\noindent where $\tilde{D}h^{\pm}$ is the diffuse part and $ D_{jump} h^{\pm}$ is the jump part. We can summarize this equality by writing that 

$$
D|h^\pm| = F(\overline{h}^\pm, h^\pm_+, h^\pm_-) Dh^\pm
$$

\noindent with $F(\overline{h}^\pm, h^\pm_+, h^\pm_-)$ a functional bounded by $1$. Using this fact, and the PDE satisfied by $h^\pm$, we get that

$$
\|h^\pm(t)\|_{L^1(\mathcal{O})} \leq -\gamma_\pm \int_{\Omega_T} F(\overline{h}^\pm, h^\pm_+, h^\pm_-) \psi  d(\partial_vf_2^\pm).
$$

This last term can be bounded by $$|\gamma_\pm||\partial_v f^\pm_2|(\Omega_T) \|\psi\|_{L^\infty(Q_T)}.$$ By the definition of  $\Lambda$ \eqref{systeme/ptfixe1}, we obtain easily that 

$$
\epsilon_{0} \| \Lambda(E_1)-\Lambda(E_2)\|_{L^\infty(Q_T)} \leq 2|\gamma_\pm||\partial_v f^\pm_2|(\Omega_T) \|E_1-E_2\|_{L^\infty(Q_T)}.
$$

Schauder theorem applies and $\Lambda$ admits a fixed point. We may now prove uniqueness separately. We consider two triplets $(E_1,f_1^+,f_1^-)$, $(E_2,f_2^+,f_2^-)$ solutions of the Vlasov-Poisson system. Using the same notations as before, $h^\pm$ satisfies the Vlasov equation with a source term, and we obtain for $t \in [0,T]$

$$
\|h^\pm(t)\|_{L^1(\mathcal{O})} \leq -\gamma_\pm \int_{ \Omega_t} \psi F(\overline{h}^\pm, h^\pm_+, h^\pm_-) d(\partial_vf_2^\pm).
$$

Thanks to  Theorem \ref{theoreme/desintegration}, we  disintegrate the measure $d(\partial_{v}f_2^\pm)$  to get 

$$
\|h^\pm(t)\|_{L^1(\mathcal{O})} \leq -\gamma_\pm \int_0^t\int_{0}^{L}\int_\mathbb{R} \psi(s,x) F(\overline{h}^\pm, h^\pm_+, h^\pm_-) d(\partial_vf_2^\pm(s)) ds.
$$

The right-hand side is easily bounded and 

$$
\|h^\pm(t)\|_{L^1(\mathcal{O})} \leq |\gamma_\pm| \int_0^t   \sup_{x\in(0,L)}|\psi(s,x)|   TV(f_2^\pm(s),\mathcal{O}) ds.
$$

Using the uniform estimates on the total variation along with the expression of $\psi$ and mass estimates, we get that 

$$
\|h^+(t)\|_{L^1(\mathcal{O})} + \|h^-(t)\|_{L^1(\mathcal{O})} \leq K \int_0^t \|h^+(s)\|_{L^1(\mathcal{O})} + \|h^-(s)\|_{L^1(\mathcal{O})}ds,
$$

\noindent where $K \geq 0$ is a constant that depends only on the data. The Grönwall states that $\|h^+(t)\|_{L^1((0,L)\times\mathbb{R})} + \|h^-(t)\|_{L^1((0,L)\times\mathbb{R})}=0 $ for every $t\in[0,T]$. This implies directly that $\psi=0$.
\end{proof}

\section{The Vlasov-Poisson-Ampère system} \label{sec:Vlasov-Poisson-floating}

We now come to the proof of the main result (Theorem \ref{MAIN_RESULT}). Thus, throughout this section, we assume the hypotheses of the main result. For the sake of conciseness, we shall omit in the following mathematical statements to specify the assumptions of regularity on the data. The proof proceed again by a fixed point argument. To prove the global in time existence, we have first to prove the existence on $[0,T_{1}]$, with $T_{1}$ small enough, and then repeat the argument on adjacent small-time intervals. As in the first step, the constants in the estimates grow as $T_{1}$ increases, we cannot repeat the process trivially. To do so, we use an energy estimates and the fact that the floating potential equation enables us to give a uniform in $T_{1}$ bound on $\phi(t,L)$ which depends only on the data. Let $0 < T_{1}  < T$ to be specified later and consider the map

\begin{align}
  \Psi : W^{1,\infty}([0,T_1])\mapsto W^{1,\infty}([0,T_1])  
\end{align}

\noindent where for a given $u\in  W^{1,\infty}([0,T_1]) $, $\Psi(u)$ is defined by 

$$
\Psi(u)(t): =\beta+\dfrac{1}{\epsilon_{0}}\int_0^t\int_{0}^{L} (q_+j^++q_-j^-)(\tau,s)ds d\tau.
$$

In this expression, the currents $j^{+}$ and $j^{-}$ are given by the solutions $f^+,f^-$ of the Vlasov-Poisson system

\begin{equation}
\left\lbrace
\begin{array}{ll}
\partial_t f^\pm + v \partial_x f^\pm +\gamma_\pm E \partial_v f^\pm = 0 &  \textnormal{ in }  \Omega_{T_1}, \\ 
f^\pm(0,x,v) = \tilde{f}^{\pm}(x,v) & \textnormal{ if } (x,v)\in \mathcal{O},  \\
f^\pm(t,0,v) = g^\pm(t,v)   & \textnormal{ if }(t,v)\in (0,T_1)\times\mathbb{R}^{+\star},\\
f^\pm(t,L,v) = 0   & \textnormal{ if } (t,v)\in (0,T_1)\times\mathbb{R}^{-\star} ,
\end{array}\right.
\end{equation}
\begin{multline}
E(t,x) = -\dfrac{1}{L\epsilon_{0}}\int_{0}^{L}\int_0^s (q_+\rho^+ +q_-\rho^-)(t,y) dyds + \frac{1}{\epsilon_{0}}\int_0^x(q_+\rho^+ +q_-\rho^-)(t,y) dy \\ -\frac{1}{L} u(t)  \mbox{ in } Q_{T_{1}}.
\end{multline}

This operator is well-defined since the solution of the Vlasov-Poisson system is such that $|v|f^{+},|v|f^{-} \in BV(\Omega_{T_1})$.
Solutions of the Vlasov-Poisson-Ampère system \eqref{m_vlasov+}-\eqref{m_initial_datum} correspond to fixed points of $\Psi$.

\begin{lemme}[Invariant set]\label{m_Invariant_Set_PSI} There are constants $C > |\beta|$, $M \geq 2C +1 \sqrt{2C}$ and $0 < T_1 \leq \min(\frac{1}{\sqrt{2C}},T)$ such that
$\Psi(K_{M,T_{1}}) \subset B_{M,T_{1}}$ where $B_{M,T_1} = \Big \lbrace u \in W^{1,\infty}(0,T_1) \  :  \ u(0)=\beta \ ; \ \|u\|_{L^\infty(0,T_1)}\leq M \ ; \   \|u'\|_{L^\infty(0,T_1)}\leq C(1+T_1(1+M))   \Big\rbrace.$
\end{lemme}

\begin{proof}
Let $u \in L^{\infty}(0,T_{1}).$ Let us recall the classic bound on $E$:
 
    \begin{multline*}
    \| E\|_{L^\infty(Q_{T_1})} \leq    \frac{\|u\|_{L^\infty((0,T_1))}}{L} \\ + 2\frac{(q_++|q_-|)}{\epsilon_{0}} ( \|\tilde{f}^+ \|_{L^1(\mathcal{O})} +  \|\tilde{f}^- \|_{L^1(\mathcal{O})} +  \|vg^+ \|_{L^1((0,T_1)\times\mathbb{R}^+)} +  \|vg^- \|_{L^1((0,T_1)\times\mathbb{R}^+)}).
    \end{multline*}
    
We provide a bound on the currents by Proposition \ref{proposition/streamtube} applied to $|v|f^\pm$ and an estimate of the total mass yields

\begin{align*}
        & \underset{ 0 \leq t \leq T_{1} } \sup \|vf^\pm(t) \|_{L^1(\mathcal{O})} \leq  \|v^2g^\pm \|_{L^1((0,T_1)\times\mathbb{R}^+)} +  \|v\tilde{f}^\pm \|_{L^1(\mathcal{O})} \\
        &+ T_1 ( \|vg^\pm \|_{L^1((0,T_1)\times\mathbb{R}^+)}+  \|\tilde{f}^\pm \|_{L^1(\mathcal{O})}) \| E\|_{L^\infty(Q_{T_1})}.
\end{align*}
Using the expression of $\Psi(u)$, we find readily that there exists $C>|\beta|$ such that 

    $$
    \| \Psi(u) \|_{L^\infty(0,T_1)} \leq C(1+T_1+T_1^2(1+\|u\|_{L^\infty(0,T_1)} )),
    $$
    $$
     \| \Psi(u)' \|_{L^\infty(0,T_1)}  \leq C(1+T_1(1+\|u\|_{L^\infty(0,T_1)} )).
     $$
     
Since $\|u\|_{L^\infty(0,T_1)}\leq M$, we see that $\| \Psi(u) \|_{L^\infty(0,T_1)}\leq M$ if $C(1+T_1+T_1^2(1+M ))\leq M$. Since $T_{1}^2 \leq \frac{1}{2C}$, this last inequality is equivalent to $M \geq \frac{C(1+T_{1} + T_{1}^2)}{(1-CT_{1}^2)}$. Now observe that the lower bound in the previous inequality is an increasing  function of the variable $T_{1}.$ So, for $T_{1} \leq \frac{1}{\sqrt{2C}}$ we have  $\frac{C(1+T_{1} + T_{1}^2)}{(1-CT_{1}^2)} \leq 2C +1 +\sqrt{2C}.$ So, it is sufficient to pick $M \geq 2C +1 +\sqrt{2C}.$

\end{proof}

\begin{lemme}[Fixed point of  $\Psi$] \label{m_Lipschitz_PSI}
    Fix the constants $C > |\beta|$, $T_{1} \leq \min(\frac{1}{\sqrt{2C}},T)$ and $ M \geq 2C +1 + \sqrt{2C}.$ Then  $\Psi: B_{M,T_1}\mapsto B_{M,T_1}$ is  Lipschitz continuous for the norm $\| \cdot \|_{L^{\infty}(0,T_{1})}$ and admits a unique fixed point.
\end{lemme}

\begin{proof}
Let $u_1,u_2$ be in $K_{M,T_1}$ and $(f_1^+,f_1^-,E_1)$ and $(f_2^+,f_2^-,E_2)$ the corresponding solutions of the Vlasov-Poisson system.
We set $h^\pm := f_1^\pm-f_2^\pm$, $\rho_h^\pm: = \int_\mathbb{R}h^\pm dv$, $j_h^\pm := \int_\mathbb{R} vh^\pm dv$ and  $\psi = E_{1}-E_{2}.$ Then, we have for $(t,x) \in Q_{T}$,

 $$
    \psi(t,x) = -\dfrac{1}{L\epsilon_{0}}\int_{0}^{L}\int_0^s (q_+\rho_h^++q_-\rho_h^-)(t,u) duds + \frac{1}{\epsilon_{0}} \int_0^x(q_+\rho_h^++q_-\rho_h^-)(t,u) du + \frac{u_1(t)-u_2(t)}{L}.
$$

The functions $h^{+}$ and $h^{-}$ satisfy a transport equation with a source term:

\begin{equation}
    \partial_t h^\pm + v \partial_x h^\pm +\gamma_+ E_1 \partial_v h^\pm = -\gamma_\pm \psi \partial_v f_2^\pm  \textnormal{ in }  \mathcal{M}(\Omega_{T_1}),\\ 
\end{equation}

\noindent where the initial and boundary conditions vanish. Let us estimate $\|h^\pm(t)\|_{L^1(\mathcal{O})}$. By an $L^1$ estimates (Proposition \ref{proposition/streamtube}), and using the chain rule (Theorem \ref{theoreme/chainrule}), we provide the inequality:

    \begin{eqnarray*}
        \|h^\pm(t)\|_{L^1(\mathcal{O})} & \leq  & -\gamma_\pm \int_{\Omega_{T_1}} F(\overline{h}^\pm,h^\pm_+,h^\pm_-)\psi d(\partial_vf_2^\pm),
    \end{eqnarray*}
    
   \noindent  where $F(\overline{h}^\pm,h^\pm_+,h^\pm_-)$ is a functional bounded by 1 in module. The measure $d(\partial_vf_2^\pm)$ (see \cite{ambrosio}, section 3.11) can be disintegrated as $d\tau\otimes d(\partial_vf_2^\pm(\tau))$, thus we can write 
   
    \begin{eqnarray*}
        \|h^\pm(t)\|_{L^1(\mathcal{O})} & \leq   & |\gamma_\pm| \int_0^t\sup_{0\leq x \leq L}|\psi(\tau,x)| TV_v(f_2^\pm(\tau),\mathcal{O}) d\tau.
    \end{eqnarray*}
    
Using directly the expression of $\psi$, we have the bound for $\tau \in [0,t]$:

    $$
    \sup_{0\leq x \leq L}|\psi(\tau,x)| \leq 2\frac{q_++|q_-|}{\epsilon_{0}}\Big( \|h^+(\tau)\|_{L^1(\mathcal{O})} +  \|h^-(\tau)\|_{L^1((\mathcal{O})}\Big) + \frac{|u_1(\tau)-u_2(\tau)|}{L}.
    $$
    
Inserting this bound in the previous inequality yields

    \begin{eqnarray*}
        \|h^\pm(t)\|_{L^1(\mathcal{O})} & \leq   & 2|\gamma_\pm| \int_0^t\dfrac{(q_++|q_-|)}{\epsilon_{0}}( \|h^+(\tau)\|_{L^1(\mathcal{O})} +  \|h^-(\tau)\|_{L^1(\mathcal{O})})  TV_v(f_2^\pm(\tau),\mathcal{O}) d\tau \\
        & + & \frac{|\gamma_\pm|}{L} \int_0^t |u_1(\tau)-u_2(\tau)|  TV_v(f_2^\pm(\tau),\mathcal{O}) d\tau.
    \end{eqnarray*}
    
Summing the inequalities obtained for each species, we obtain 

     \begin{eqnarray*}
        \|h^+(t)\|_{L^1(\mathcal{O})} +  \|h^-(t)\|_{L^1(\mathcal{O})} & \leq   &  2\dfrac{(q_++|q_-|)}{\epsilon_{0}}\int_0^t( \|h^+(\tau)\|_{L^1(\mathcal{O})} +  \|h^-(\tau)\|_{L^1(\mathcal{O})}) k(\tau) d\tau \\
        & + & \frac{1}{L} \int_0^t |u_1(\tau)-u_2(\tau)|  k(\tau)  d\tau,
    \end{eqnarray*}
    
\noindent where $k(\tau) =  |\gamma_+|TV_v(f_2^+(\tau),\mathcal{O})+ |\gamma_-|TV_v(f_2^-(\tau),\mathcal{O})$. Using Grönwall and Hölder inequalities allow us to conclude that for $t \in [0,T_{1}]$

     \begin{eqnarray*}
        \|h^+(t)\|_{L^1(\mathcal{O})} +  \|h^-(t)\|_{L^1(\mathcal{O})} & \leq   & \frac{1}{L}\|u_1-u_2\|_{L^\infty(0,t)}(|\gamma_+| |\partial_vf_2^+(\Omega_{T_1})| +   |\gamma_-| |\partial_vf_2^-(\Omega_{T_1})| ) \\
         & \times & e^{2\frac{q^{+} + |q^-|}{\epsilon_{0}}(|\gamma_+| |\partial_vf_2^+(\Omega_{T_1})| +   |\gamma_-| |\partial_vf_2^-(\Omega_{T_1})| )}.
    \end{eqnarray*}
    
    So there exists a  constant $C_{T_1}>0$,  such that for every $t \in [0,T_{1}],$
    \begin{equation}
    \|h^+(t)\|_{L^1(\mathcal{O})} +  \|h^-(t)\|_{L^1(\mathcal{O})} \leq    C_{T_1} \|u_1-u_2\|_{L^\infty(0,t)}.
    \label{equation/1}        
    \end{equation}
    
We also need an estimate on the function $vh^{\pm}$. Using Proposition \ref{proposition/streamtube} applied to  $|vh^\pm|$ gives the bound
\begin{eqnarray*}
        \|vh^\pm(t)\|_{L^1(\mathcal{O})} & \leq  &  \int_{\Omega_t} \begin{pmatrix}
            1\\
            v\\
            \gamma_+ E_1
        \end{pmatrix}\cdot d(D(|v||h^\pm|)).
\end{eqnarray*}

Using the Leibniz rule to expand $D(|v| |h^\pm|)$ and the chain rule, we find that 
$$
D(|v| |h^\pm|) = |v| F(\overline{h}^\pm,h^\pm_+,h^\pm_-) Dh^\pm + |h^\pm| sgn(v) \overrightarrow{e_v} dtdxdv,
$$
where $F(\overline{h}^\pm,h^\pm_+,h^\pm_-)$ is a functional bounded by 1 in module. We then write 
\begin{eqnarray*}
        \|vh^\pm(t)\|_{L^1(\mathcal{O})} & \leq  & -\gamma_\pm \int_{ \Omega_t}F(\overline{h}^\pm,h^\pm_+,h^\pm_-)\psi |v|d(\partial_vf_2^\pm) + \gamma_\pm \int_{\Omega_t}|h^\pm|sgn(v)E_1 dtdxdv.
\end{eqnarray*}
We then disintegrate the measure $d(\partial_{v}f_{2}^{\pm})$ (Theorem \ref{theoreme/desintegration}) and use a crude upper bound to get for $t \in [0,T_{1}]$

    \begin{equation}
        \|vh^\pm(t)\|_{L^1(\mathcal{O})}  \leq   |\gamma_\pm| \|\psi\|_{L^\infty(Q_t)} |v\partial_vf_2^\pm|(\Omega_{T_1}) 
        + |\gamma_\pm| t \sup_{0 \leq\tau\leq t}\|h^\pm(\tau)\|_{L^1(\mathcal{O})} \|E_1\|_{L^\infty(Q_{T_1})}.
        \label{equation/2}
    \end{equation}

Finally, we get from the expression of $\psi$ that 

    $$
    \|\psi\|_{L^\infty(Q_t)} \leq 2\dfrac{(q_++|q_-|)}{\epsilon_{0}}\sup_{0 \leq \tau\leq t}( \|h^+(\tau)\|_{L^1(\mathcal{O})} +  \|h^-(\tau)\|_{L^1(\mathcal{O})}) + \frac{\|u_1-u_2\|_{L^\infty(0,t)}}{L}.
    $$
    
Therefore, combining equations \eqref{equation/1} and \eqref{equation/2}, we find that there exists a $C_{T_1}>0$ which depends only on $\tilde{f}^\pm,g^\pm, M,L$ and $T_1$ such that  for any pair $(u_{1},u_{2})$ of $K_{M,T_{1}}$ we have

\begin{equation}
    \|\Psi(u_1)-\Psi(u_2)\|_{L^\infty(0,T_1)}\leq C_{T_1} \|u_1-u_2\|_{L^\infty(0,T_1)}. \label{m_Lipschitz_estimate_PSI}
\end{equation}

By Lemma \ref{m_Invariant_Set_PSI}, we have that $\Lambda$ is continuous on $B_{M,T_{1}}$. Moreover the set $B_{M,T_{1}}$ is compact for the $L^{\infty}(0,T_{1})$ topology. So, the Schauder fixed point theorem states that $\Psi$ has a fixed point, thus existence of a solution on $[0,T_{1}]$ is proven. To show uniqueness, let us consider two fixed points $u_1$ and $u_2$. Keeping the same notations as before, we get by definition of $u_1,u_2$ that for $t \in [0,T_{1}]$
\begin{eqnarray*}
\|u_1-u_2\|_{L^\infty(0,t)} &\leq &  \dfrac{q_++|q_-|}{\epsilon_{0}}\int_0^t \|vh^+(\tau)\|_{L^1(\mathcal{O})} + \|vh^-(\tau)\|_{L^1(\mathcal{O})}d\tau \\
&\leq &  \dfrac{C_{T_1}}{\epsilon_{0}}\int_0^t \|u_1-u_2\|_{L^\infty(0,\tau)}d\tau.
\end{eqnarray*}

Once again, we conclude by Gronwall inequality that $u_1=u_2$. This implies, by uniqueness for the 
 Vlasov-Poisson system, that $h^\pm=0$ and $\psi=0$. 
\end{proof}
We would like to repeat the previous argument on an interval $[T_1,T_2]$ for a suitable $T_2 \leq T$, and so on. This must be done carefully: indeed, the initial value $\beta=\phi(0,L)$ was given, but on $[T_1,T_2]$ this value is replaced by $\phi(T_1,L) = \beta + \dfrac{1}{\varepsilon_0} \int_0^{T_{1}}\int_{0}^{L} (q_+j^++q_-j^-)(\tau,x)dxd\tau$. Repeating the estimates of Lemma \ref{m_Invariant_Set_PSI} would yield constants that grows polynomially with $T_{1}$. Thus, it would not be not trivial to fix the constants $C, T_2$ and $M$ in Lemma \ref{m_Invariant_Set_PSI}. We rather use an energy estimate to get the desired bound on $\phi(t,L)$. The cornerstone, is the following energy estimates, which is obtained thanks to our nonlinear boundary condition \eqref{m_potential_bc}. Namely, the difference of potential is always less than the energy of the system. We define the energy of the system for $t \in [0,T_{1}]$ by
\begin{equation}
    \mathcal{E}(t) = \frac{m^{+}}{2} \int_{\mathcal{O}} v^2f^{+} dxdv + \frac{m^{-}}{2} \int_{\mathcal{O}} v^2f^{-}dxdv + \frac{\epsilon_{0}}{2}\int_{0}^{L}  | \partial_{x} \phi |^2 dx.
\end{equation}

\begin{proposition}[Bounds on the floating potential]

    Assume that $|v|^kf^\pm\in BV(\Omega_{T_1})$, for $k=0,1,2,3$. Then $\mathcal{E}\in BV((0,T_1))$, and 
    
    \begin{equation*}
    \dfrac{d}{dt}\mathcal{E}(t)  = \sum_\pm  -\dfrac{m_\pm}{2} \int_\mathbb{R} v^3 Tf^\pm(t,L,v)dv +  \dfrac{m_\pm}{2} \int_\mathbb{R} v^3 Tf^\pm(t,0,v)dv .
\end{equation*}
The energy obeys the bound for $t\in (0,T_1),$
\begin{equation}\label{m_floating_p_bound_1}
    \mathcal{E}(t)\leq \mathcal{E}(0) + \dfrac{m_+}{2} \|v^3g^+\|_{L^1((0,T_1)\times\mathbb{R}^+)}+ \dfrac{m_-}{2} \|v^3g^-\|_{L^1((0,T_1)\times\mathbb{R}^+)}.
\end{equation}

Eventually, the energy controls the floating potential:

\begin{equation}\label{m_floating_p_bound_2}
 \sqrt{\dfrac{2L}{\epsilon_{0}}\mathcal{E}(t) }\geq |\phi(t,L)|.
\end{equation}

\label{proposition/energieestimation}

\end{proposition}

\begin{proof}

We first prove that $\mathcal{E}$ is in $BV(0,T_1)$. Let $\phi\in C_0^\infty(0,T_1)$. First, we focus on the kinetic energy $e^\pm(t)=\int_\mathcal{O}v^2 f^\pm(t,x,v) dxdv$. By integration by parts (see Theorem \ref{theoreme/regularite_trace}), the Vlasov equation and the Leibniz rule ($v^2\partial_vf = \partial_v(v^2f)-2vf dtdxdv$)  we have that



\begin{eqnarray*}
 \int_0^T e^\pm(t) \partial_t\phi(t) dt  &=& \int_{\Omega_{T_1}} v^2 f^\pm \Div (\phi(t),0,0) dtdxdv  \\
					& =&  -\int_{\Omega_{T_1}} v^2 \phi(t) d(-v\partial_x f^\pm -\gamma_\pm \partial_v f^\pm) \\
					& = & \int_{\Omega_{T_1}} v^3 \phi(t) d(\partial_x f^\pm) + \gamma_\pm \int_{\Omega_{T_1}} v^2 E \phi(t) d(\partial_v f^\pm ) \\
					& = & \int_0^{T_1}\phi(t)\int_\mathbb{R} v^3Tf^\pm(t,L,v)dv dt - \int_0^{T_1}\phi(t)\int_\mathbb{R}v^3 Tf^\pm(t,0,v)dv dt \\
					&-& 2\gamma_\pm \int_{\Omega_{T_1}}Evf^\pm \phi(t) dtdxdv.
\end{eqnarray*}

Thus, in the distributional sense, 

\begin{multline}
\sum_\pm \frac{m_\pm}{2} \frac{d}{dt}e^\pm(t) = \sum_\pm \left(-\frac{m_\pm}{2}\int_\mathbb{R} v^3Tf^\pm(t,L,v)dv + \frac{m_\pm}{2} \int_\mathbb{R}v^3 Tf^\pm(t,0,v)dv \right) \\+\int_0^L E (q_+j^+ + q_- j^-) dx.
\end{multline}

Remember that integrals of the trace are well-defined by Theorem \ref{theoreme/regularite_trace}.  Using Ampère equation, the last term can be rewritten as $-\varepsilon_0 \int_0^L \partial_t\partial_x\phi \partial_x\phi dx$. The potential energy is easier to work with, and by classic theory of Sobolev spaces, its derivative is exactly $\varepsilon_0 \int_0^L \partial_t\partial_x\phi \partial_x\phi dx$. By cancellation of the latter terms, we get that

    \begin{equation*}
    \dfrac{d}{dt}\mathcal{E}(t)  = \sum_\pm  -\dfrac{m_\pm}{2} \int_\mathbb{R} v^3 Tf^\pm(t,L,v)dv +  \dfrac{m_\pm}{2} \int_\mathbb{R} v^3 Tf^\pm(t,0,v)dv .
\end{equation*}

The bound on $\mathcal{E}$ is found by integration. Indeed, we have that 

    \begin{equation}
    \mathcal{E}(t)-\mathcal{E}(0)  =  \sum_\pm -\dfrac{m_\pm}{2} \int_0^{t}\int_\mathbb{R} v^3 Tf^\pm(s,L,v)dvds +  \dfrac{m_\pm}{2}\int_0^{t} \int_\mathbb{R} v^3 Tf^\pm(s,0,v)dv ds .
\end{equation}

By simple sign considerations and by using the expression of the trace of $f$ on the incoming part of the boundary, we find the desired bound. Finally, the last part stems from the triangle and Jensen inequalities:

    $$
\mathcal{E}(t) \geq \dfrac{\epsilon_{0}}{2}\int_{0}^{L} |\partial_x\phi|^2 dx \geq \dfrac{\epsilon_{0}}{2L}\left(\int_{0}^{L} |\partial_x\phi| dx\right)^2 \geq \dfrac{\epsilon_{0}}{2L}\left(\int_{0}^{L} \partial_x\phi dx \right)^2 = \dfrac{\epsilon_{0}}{2L} \phi(t,L)^2.
$$

\end{proof}

The bounds \eqref{m_floating_p_bound_1}, \eqref{m_floating_p_bound_2} show that $\phi(T_{1},L)$ only depends on the data. So we can repeat the previous estimates of Lemma  \ref{m_Invariant_Set_PSI}, \ref{m_Lipschitz_PSI} on $\Psi$. Thus, the solution can be extended to $[0,T]$ using the Schauder fixed point theorem. It concludes the proof of the Theorem \ref{MAIN_RESULT}.

\section{Conclusion}

This work shows the well-posedness of the Vlasov-Poisson-Ampère system, by studying the Vlasov-Poisson system with floating potential at the wall. Several extensions could be addressed in the future. First, regularity  may be relaxed. Indeed, we believe that the BV regularity is not necessary as it is only used to establish energy and $L^1$ estimates.  A second extension may study the long-time asymptotic of this system. This would be particularly interesting as stationary plasma sheath models are widely used in physics.  
A major difficulty to show the nonlinear stability of this system by the relative entropy method (see \cite{BENABDALLAH2000867}) is that the boundary condition breaks the structure needed to get dissipation and convexity of the relative entropy functional. We believe that a trajectorial approach as in \cite{Badsi-25_mybib} should work.

\section{Appendix}
\subsection{Proof of Theorem \ref{theoreme:trace}}
\begin{proof}
Let us define the trace as before. Indeed, $f$ lies in $BV((0,T)\times(0,L)\times(-B,B))$ for every $B\geq 0$, so its trace is defined for almost every  $(t,x,v)\in \partial \Omega_T$ by  

$$
 Tf(t,x,v) := \lim\limits_{R\to 0^+} \fint_{\Omega_T\cap B((t,x,v),R)} f(t',x',v') dt'dx'dv'. 
$$

Remark that the trace is nonnegative. We first  show that   $T(|v|^k f) = |v|^k Tf$. Indeed, we have that for a fixed $v$, 

\begin{multline}
\left|\fint_{\Omega_T\cap B((t,x,v),R)} (|v'|^k-|v|^k) f(t',x',v') dt'dx'dv'\right| = \\ \left|\sum_{s=0}^{k-1} \fint_{\Omega_T\cap B((t,x,v),R)} (|v'|-|v|) |v'|^{k-1-s}|v|^s f(t',x',v') dt'dx'dv'\right| .
\end{multline}

This last term is easily bounded. Indeed, $f\in L^\infty(\Omega_T)$, $\left||v'|-|v|\right|\leq R$, $|v'|\leq |v|+R$ and $v$ is fixed. There exists a constant $C_v$ depending on $v$ such that

$$
\fint_{\Omega_T\cap B((t,x,v),R)} (|v'|^k-|v|^k) f(t',x',v') dt'dx'dv'  \leq C_v \|f\|_{L^\infty(\Omega_T)} R .
$$

The bound goes to $0$ as $R\to 0$. We thus conclude by using the last estimate, the definition of the trace, and by rearranging the difference $(T(|v|^k f) -|v|^k Tf)(t,x,v)$ as following:

\begin{multline}
(T(|v|^k f) -|v|^k Tf)(t,x,v) = T(|v|^k f)(t,x,v) - \fint_{\Omega_T\cap B((t,x,v),R)} |v'|^k f dtdxdv \\ +  \fint_{\Omega_T\cap B((t,x,v),R)} |v'|^k f dtdxdv -  |v|^k \fint_{\Omega_T\cap B((t,x,v),R)}f dtdxdv \\+ |v|^k \fint_{\Omega_T\cap B((t,x,v),R)} f dtdxdv - |v|^k Tf(t,x,v).
\end{multline}

Now,  consider  $\phi_1,\phi_2\in C^\infty(Q_T)$ and $\psi_n$ a sequence of smooth functions converging uniformly toward $\psi$ on $[-R,R]$, with $R>0$. We apply the integration by parts formula to the function $f$ and the field $\chi_R(v)\psi_n(v)(\phi_1,\phi_2,0)$ to obtain

$$
\int_{\Omega_T} f\chi_R \psi_n \Div(\phi_1,\phi_2,0) dtdxdv = \int_{\partial\Omega_T} Tf \chi_R\psi_n(\phi_1,\phi_2,0)\cdot \overrightarrow{n} d\mathcal{H}^2-\int_{\Omega_T}\chi_R \psi_n (\phi_1,\phi_2,0)\cdot d(Df).
$$

By dominated convergence, every term converges as $n\to\infty$, and we get 

\begin{multline}
\int_{\Omega_T} f \psi \chi_R  \Div(\phi_1,\phi_2,0) dtdxdv = \int_{\partial\Omega_T} Tf \psi \chi_R(\phi_1,\phi_2,0)\cdot \overrightarrow{n} d\mathcal{H}^2\\-\int_{\Omega_T} \psi \chi_R(\phi_1,\phi_2,0)\cdot d(Df).
\end{multline}

We now show that $|v|^k Tf(t,L,v)\in L^1((0,T)\times\mathbb{R})$ if $|v|^k f\in BV(\Omega_T)$ (the other faces are similar). By applying the formula above to $\phi_1=0,\phi_2(x)=x, \psi(v)=|v|^k$, we get the estimate 

$$
\int_0^T\int_{\mathbb{R}} |v|^k \chi_R Tf(t,L,v) dtdv \leq \| v^k f\|_{L^1(\Omega_T)} + L|\partial_x(|v|^k f)|(\Omega_T),
$$

which implies the integrability of the trace  by monotone convergence.  Knowing this, we can take the limit $R\to +\infty$ (with $\psi(v)=v^k$) in the integration by part formula to obtain

\begin{multline}
\int_{\Omega_T} f v^k  \Div(\phi_1,\phi_2,0) dtdxdv = \int_{\partial\Omega_T} Tf v^k(\phi_1,\phi_2,0)\cdot \overrightarrow{n} d\mathcal{H}^2\\-\int_{\Omega_T} v^k(\phi_1,\phi_2,0)\cdot d(Df).
\end{multline}

Remark that the dominated convergence applies since $|v|^kTf\in L^1(\partial\Omega_T)$, and that the last integrand can be rewritten as $\chi_R(\phi_1,\phi_2,0)\cdot d(D(v^kf))$, with $D(v^kf)$ a bounded Radon measure and $\chi_R(\phi_1,\phi_2,0)$ continuous and bounded.

\end{proof}

\bibliographystyle{plain}
\bibliography{biblio}

\end{document}